\documentclass{article}
\pdfoutput=1

\usepackage{amsmath,amsthm,amssymb}
\usepackage{mathrsfs}
\usepackage{graphicx}
\usepackage{amsmath}
\usepackage{dsfont}
\usepackage{multirow}
\usepackage{graphicx}
\usepackage{epstopdf}
\usepackage{geometry}
\usepackage{color}
\usepackage{ulem}
\usepackage{hyperref}
\usepackage{enumerate}
\usepackage{caption}
\usepackage{subcaption}

\usepackage{authblk}

\title{Quasi-Monte Carlo sampling for machine-learning partial differential equations}
\author[a,b]{Jingrun Chen \thanks{jingrunchen@suda.edu.cn}}
\author[a,b]{Rui Du \thanks{durui@suda.edu.cn}}
\author[a]{Panchi Li \thanks{LiPanchi1994@163.com}}
\author[a,c]{Liyao Lyu \thanks{lylv@stu.suda.edu.cn}}
\affil[a]{School of Mathematical Sciences, Soochow University, Suzhou, 215006, China.}
\affil[b]{Mathematical Center for Interdisciplinary Research, Soochow University, Suzhou, 215006, China.}
\affil[c]{CW Chu College, Soochow University, Suzhou, 215006, China.}

\date{2019/10/21}

\graphicspath{{figure/}}

\chardef\bslash=`\\ 

\newtheorem{thm}{Theorem}[section]

\newtheorem{lem}[thm]{Lemma}

\newtheorem{assumption}[thm]{Assumption}

\theoremstyle{definition}

\theoremstyle{remark}
\newtheorem{rem}{Remark}[section]

\newcommand{\thmref}[1]{Theorem~\ref{#1}}
\newcommand{\secref}[1]{\S\ref{#1}}
\newcommand{\lemref}[1]{Lemma~\ref{#1}}

\newcommand{\eval}[2][\right]{\relax
  \ifx#1\right\relax \left.\fi#2#1\rvert}

\let\abs=\envert

\let\norm=\enVert

\begin{document}
\maketitle
\renewcommand{\sectionmark}[1]{}

\begin{abstract}
Solving partial differential equations in high dimensions by deep neural network has brought significant attentions in recent years. In many scenarios, the loss function is defined as an integral over a high-dimensional domain. Monte-Carlo method, together with the deep neural network, is used to overcome the curse of dimensionality, while classical methods fail. Often, a deep neural network outperforms classical numerical methods in terms of both accuracy and efficiency. In this paper, we propose to use quasi-Monte Carlo sampling, instead of Monte-Carlo method to approximate the loss function. To demonstrate the idea, we conduct numerical experiments in the framework of deep Ritz method proposed by Weinan E and Bing Yu \cite{E2018}. For the same accuracy requirement, it is observed that quasi-Monte Carlo sampling reduces the size of training data set by more than two orders of magnitude compared to that of MC method. Under some assumptions, we prove that quasi-Monte Carlo sampling together with the deep neural network generates a convergent series with rate proportional to the approximation accuracy of quasi-Monte Carlo method for numerical integration. Numerically the fitted convergence rate is a bit smaller, but the proposed approach always outperforms Monte Carlo method. It is worth mentioning that the convergence analysis is generic whenever a loss function is approximated by the quasi-Monte Carlo method, although observations here are based on deep Ritz method.
\end{abstract}



\section{Introduction}
\label{intro}

Deep neural networks (DNNs) have had great success in text classification, computer vision, natural language processing and other data-driven applications (see, e.g., \cite{Goodfellow2016,Wang2012EndtoendTR,NIPS2012_4824,hinton2012deep,Sarikaya:2014:ADB:2687012.2687014}). Recently, DNNs have been applied to the field of numerical analysis and scientific computing, with the emphasis on solving high-dimensional partial differential equations (PDEs) (see, e.g., \cite{E2018,Jiequn2018,Fan2019,deepGalerkin2018}), which are widely used in physics and finance. Notable examples include Schr\"odinger equation in the quantum many-body problem \cite{Carleo2017,Han2019}, Hamilton-Jacobi-Bellman equation in stochastic optimal control \cite{Jiequn2018,Weinan2017349}, and nonlinear Black-Scholes equation for pricing financial derivatives \cite{Beck20191563,GonzlezCervera2019}.

Classical numerical methods, such as finite difference method \cite{Randall2007} and finite element method \cite{FEM2007}, share the similarity that the approximation stencil has compact support, resulting in the sparsity of stiffness matrix (or Hessian in the nonlinear case). Advantages of these methods are obvious for low dimensional PDEs (the dimension $K\leq 3$). However, the number of unknowns grows exponentially as $K$ increases and classical methods run into the curse of dimensionality. In another line, spectral method \cite{spectral2011} uses basis functions without compact support and thus sacrifices the sparsity, but often has the exponential accuracy. However, the number of modes used in the spectral method also grows exponentially as $K$ increases. Sparse grid method \cite{gerstner_numerical_1998,bungartz_sparse_2004} mitigates the aforementioned situation to some extent ($K\leq 9$ typically). Therefore, high-dimensional PDEs are far out of the capability of classical methods.

The popularity of DNNs in scientific computing results from its ability to approximate a high-dimensional function without the curse of dimensionality. To illustrate this, we focus on methods in which the loss function is defined as an integral over a bounded domain in high dimensions; see the deep Ritz method \cite{E2018} and the deep Galerkin method \cite{DG2016} for example. The success of DNNs relies on composition of functions without compact support and sampling strategy for approximating the high-dimensional integral. It is known that the choice of approximate functions in DNNs is of particular importance. For example, in the current work, the approximate function in one block of DNN consists of two linear transformations, two nonlinear activation functions, and one shortcut connection. Besides, since the network architecture is chosen a priori, thus the number of parameters can be independent of $K$ or only grows linearly as $K$ increases. On the other hand, only a fixed number of samples (or at most linear growth) is used to approximate the high-dimensional integral. Altogether, DNNs can overcome the curse of dimensionality when solving high-dimensional PDEs. In \cite{Bottou2018optimization}, the above step of numerical quadrature is viewed as approximating the expected risk by its empirical risk using Monte Carlo (MC) method. Consequently, the full gradient of the loss function is approximated by a finite number of samples and the stochastic gradient descent (SGD) method is used to find the optimal set of parameters in the network. It is shown that such a procedure converges under some assumptions \cite{Bottou2018optimization}.

From the perspective of numerical analysis, using $N$ i.i.d. random points, MC method approximates an integral with $\mathcal{O}(N^{-1/2})$ error \cite{liu2008monte}. It is also known that using $N$ carefully chosen (deterministic) points, quasi-Monte Carlo (QMC) method approximates an integral with $\mathcal{O}((\log{N})^K/N^{-1})$ error and the logarithmic factor can be removed under some assumptions \cite{Ogata1989,QMCerror1992,Dick2013QMC,Bruno2004survey}. Therefore, it is natural to replace MC method by QMC method in the community of machine learning. One example is the usage of QMC method in variational inference and QMC method has been proved to perform better than MC method \cite{QMCvariational2018}. Another example is the usage of QMC sampling in the stage of data generation for training DNNs; see an application in organic semiconductors \cite{Liyao2019}. In this work, we consider another application of QMC method, i.e., approximation of the high-dimensional integral in machine-learning PDEs.

In order to demonstrate the advantages of QMC method, we take deep Ritz method \cite{E2018} as an example. Results obtained here shall be applicable to other methods, like deep Galerkin method \cite{DG2016}, where a high-dimensional integral is defined as the loss function. Briefly speaking, deep Ritz method solves a variational problem coming from a high-dimensional PDE using a deep neural network with residual connection. Data are drawn randomly over the high-dimensional domain to train the parameters of the neural network. All numerical observations in the current work are based on deep Ritz method. In deep Galerkin method, the loss function contains not only the volume integral over the high-dimensional domain but also penalty terms for boundary conditions and initial conditions. We also demonstrate the advantage of QMC method in deep Ritz method when the penalty term is present. Theoretically, under certain assumptions, we prove a convergence result of the SGD method with respect to both the iteration number and the size of training data set.

The paper is organized as follows. We first introduce deep Ritz method for PDEs and QMC method in \secref{sec:QMC for deep-ritz}. Numerical results of QMC sampling and MC sampling are shown in \secref{sec:numerical results} with convergence analysis given in \secref{sec:analysis}. Conclusions are drawn in \secref{sec:conclusion}.

\section{Quasi-Monte Carlo sampling for deep Ritz method}
\label{sec:QMC for deep-ritz}

For completeness, we first introduce deep Ritz method. The basic idea is to solve a variational problem associated to a PDE using DNNs. The training data points are chosen randomly over the given domain using MC method. SGD method is then used to find an optimal solution. In the current work, QMC method is employed to replace MC method and the other components are remained almost the same. For consistency, we use superscripts for indices of sampling points and subscripts for coordinates of a vector throughout the paper.

\subsection{Loss function}
We take the variational problem associated to the Poisson equation \cite{evans_2010} as an example
\begin{equation}\label{equ:functional problem}
\min_{u\in H}I[u],
\end{equation}
where the loss function (objective function) $I[u]$ reads as
\begin{equation}\label{equ:loss function}
  I[u]=\int_{\Omega}\left(\frac{1}{2}\left|\nabla u(x)\right|^2-f(x)u(x)\right) \mathrm{d}x,
\end{equation}
and the set of trial functions $H$ is of infinite dimension. Here $f$ is a given function, representing the external force to the system and $\Omega$ is a bounded domain in $\mathbb{R}^K$.

When the solution of a PDE is approximated by a neural network $u(x)\approx\hat{u}_{\theta}(x)$, i.e., $H$ is restricted to a finite-dimensional space,
our goal is to find the optimal set of parameters in the neural network, denoted by $\theta$, such that
\begin{equation}\label{equ:dnnsolution}
I(\theta)=\int_{\Omega}\left(\frac{1}{2}\left|\nabla\hat{u}_{\theta}(x)\right|^2 -f(x)\hat{u}_{\theta}(x)\right) \mathrm{d}x
\end{equation}
is minimized. Numerically, a quadrature scheme is needed and the above objective function is approximated by
\begin{equation}\label{equ:discretize}
I(\theta)\approx \frac{1}{N}\sum_{i=1}^{N}\left(\frac{1}{2}\left|\nabla\hat{u}_{\theta}(x^i)\right|^2 -f(x^i)\hat{u}_{\theta}(x^i)\right).
\end{equation}
with $N$ sampling points $\{x^i\}_{i=1}^N$ which will be specified in \secref{sec:sampling}.

\subsection{Trail function and network architecture}\label{sec:QMC for deep-ritz sub:trail function}

The neural network we use here is stacked by several blocks with each containing two linear transformations, two activation functions and one shortcut connection. The $i$-th block can be formulated as
\begin{equation}\label{equ:ith block}
  t= f_i(s)= \sigma(\theta_{i,2}\cdot\sigma(\theta_{i,1}\cdot s +b_{i,1})+b_{i,2})+s.
\end{equation}
Here $s\in R^{m}$  is the input, $t\in R^{m}$  is the output, weights $\theta_{i,1}, \theta_{i,2} \in R^{m\times m}, b_{i,1}, b_{i,2}\in R^{m}$ and $\sigma$ is the activation function. Figure \ref{fig:resnet} demonstrates one block of the network.
\begin{figure}[ht]
	\includegraphics[width=\textwidth]{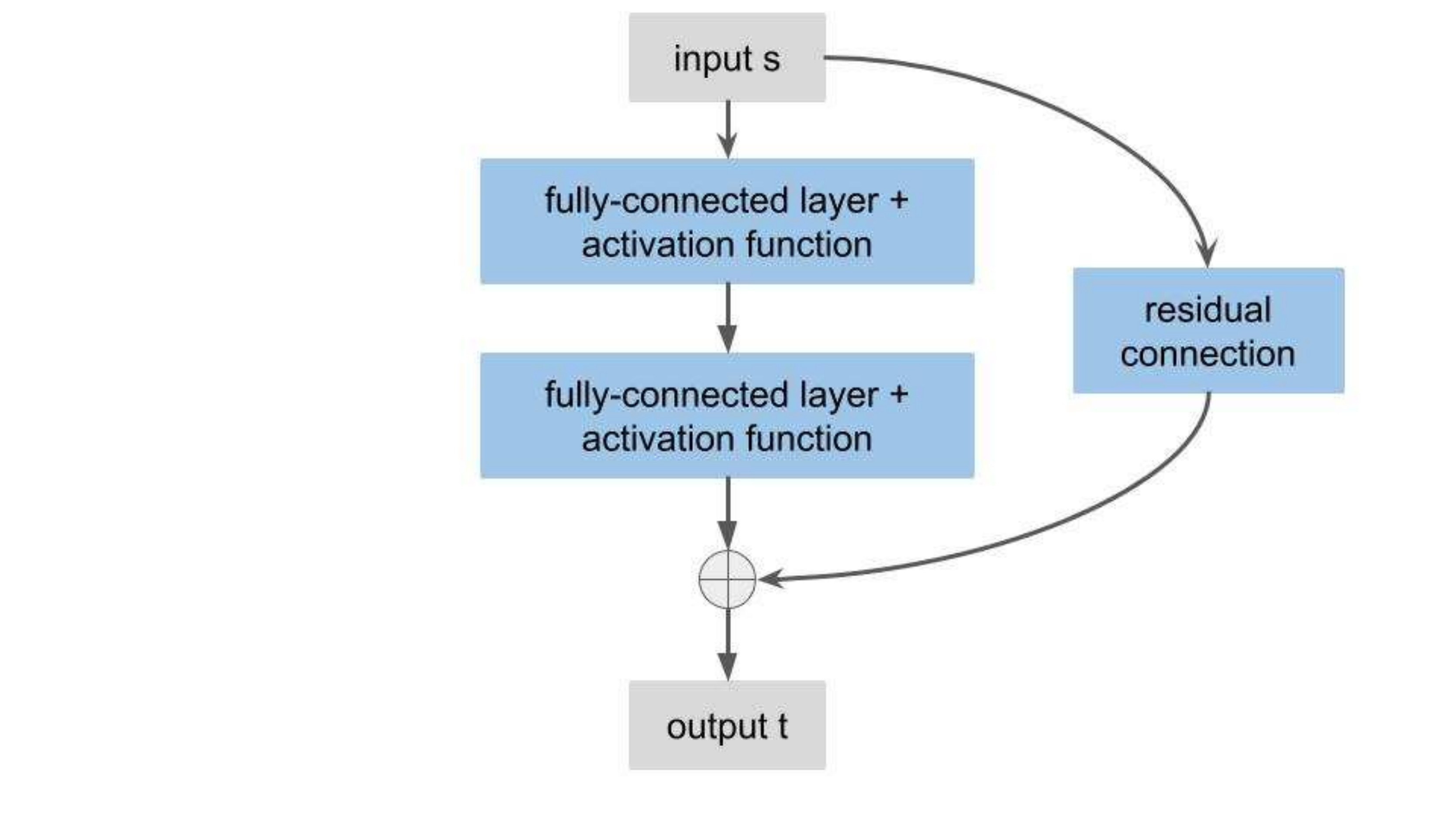}
	\caption{One block of the network. Typically a deep neural network contains a sequences of blocks, each of which consists of two fully-connected layer and one shortcut connection.}
	\label{fig:resnet}
\end{figure}

To balance simplicity and accuracy, we use the following swish function
\begin{equation}\label{equ:activation function}
  \sigma(x)=\frac{x}{1+\exp(-x)}
\end{equation}
as the activation function \cite{DBLP:journals/corr/abs-1710-05941}, which is different from the one used in \cite{E2018}.

The last term on the right-hand side of \eqref{equ:ith block} is called the shortcut connection or residual connection.
Benefits of using it are \cite{DBLP:journals/corr/HeZRS15}:
\begin{enumerate}[1)]
\item It can solve the notorious problem of vanishing/exploding gradient automatically;
\item Without adding any parameters or computational complexity, the shortcut connection performing as an \textit{Identity} mapping can resolve the degradation issue (with the network depth increasing, accuracy gets saturated and then degrades rapidly).
\end{enumerate}
With these components, the fully connected $n$-layer network can be expressed as
\begin{equation}\label{equ:fully network}
  f_{\theta'}(x)= f_n \circ f_{n-1} \cdots \circ f_1(x),
\end{equation}
where $\theta'$ denotes the set of parameters in the network. Since the input $x$ in the first block is in $\mathbb{R}^K$, not in $\mathbb{R}^m$, we need to apply a linear transformation on $x$ before putting it into the network structure. Having $f_{\theta'}(x)$, we obtain $\hat{u}_{\theta}(x)$ by
\begin{equation}\label{equ:approximate u}
  \hat{u}_{\theta}(x)= a \cdot f_{\theta'}(x) +b,
\end{equation}
where $\theta =\{\theta',a,b\}$. Note that the parameters $a$ and $b$ in \eqref{equ:approximate u} also need to be trained.

To make \eqref{equ:functional problem} - \eqref{equ:loss function} have a unique solution, a boundary condition has to be imposed. Consider the inhomogeneous Dirichlet boundary condition for example
\begin{equation}\label{equ:DBC problem}
\begin{aligned}
&u(x) = g(x) &    x \in\partial \Omega.
\end{aligned}
\end{equation}
One way to implement it is to select trail functions that satisfy the boundary condition and thus have to be problem-dependent. To avoid this, we build trail functions of the form
\begin{equation}\label{equ:build for DBC}
\hat{u}_{\theta}(x)  = A(x)\cdot( a\cdot f_{\theta'}(x) +b) +B(x), \quad x\in \Omega,
\end{equation}
where by choice $A(x) =0$ when $x\in \partial\Omega$ and $B(x)=g(x)$ when $x \in \partial \Omega$. Therefore, $\hat{u}_{\theta}(x)$ satisfies the boundary condition automatically. For Neumann boundary condition, however, we have to add a penalty term into the loss function; see \eqref{equ:loss for general} in \secref{sec:NBC} for example.

\subsection{Sampling strategies}\label{sec:sampling}

The loss function is defined over a high-dimensional domain, thus only a fixed size of points (mini-batch $\{x^i\}_{i=1}^{N}$) is allowed to approximate the integral. Due to the curse of dimensionality, standard quadrature rules may run into the risk where the integrand is minimized on fixed points but the functional itself is far away from being minimized \cite{E2018}. Therefore, points are chosen randomly and the approximation accuracy is of $\mathcal{O}(N^{-1/2})$ \cite{liu2008monte}. For stochastic problems, from the perspective of sampling strategies, it is well known that QMC method performs much better with the same size of sampling point \cite{niederreiter1992random,Russel1998,dick2013high}. We briefly review both methods here.

Consider $\Omega = [0, 1]^K$ ($K\gg1$) for convenience and let $u$ be an integrable function in $\Omega$
\begin{equation}
  \label{equ:general integeral}
  I(u) = \int_{\Omega}u(x)\mathrm{d}x < \infty,
\end{equation}
which is approximated by $N$ points of the form
\begin{equation}
  \label{equ:integration rule}
  Q_N(u) = \frac 1{N}\sum_{i=1}^{N} u(x^i).
\end{equation}
Let $P_N = \{x^i\}_{i = 1}^N\subset\Omega$ being the prescribed sampling points. In MC method, these points are chosen randomly and independently from the uniform distribution in $\Omega$. There exists a probabilistic error (root mean square error) estimate for MC method
\[
\sqrt{\mathds{E}\left[ \abs{I(u) - Q_N(u)}^2 \right]} = \frac{\sigma(u)}{\sqrt{N}},
\]
where $\sigma^2(u)$ is the variance of $u$ of the form
\[
\sigma^2(u) = I(u^2) - (I(u))^2.
\]
It is easy to check that MC method is unbiased, i.e., $\mathds{E}[Q_N(u)] = I(u)$. The variance of MC method is
\begin{align*}
  \textrm{Var}(Q_N(u)) &= \mathds{E}\left[ \abs{I(u) - Q_N(u)}^2 \right] \\ &= \frac 1{N(N-1)}\sum_{i=0}^{N-1}(u(x^i) - Q_N(u))^2\\ &= \frac 1{N(N-1)}\left(\sum_{i=0}^{N-1}u^2(x^i) - N[Q_N(u)]^2 \right).
\end{align*}

In QMC method, however, sampling points are chosen in a deterministic way to approximate the integral with the best approximation accuracy; see for example \cite{Ogata1989,QMCerror1992,Dick2013QMC,Bruno2004survey}. The deterministic feature of QMC method leads to a guaranteed error bounds and faster convergence rate for smooth integral functions. More explicitly, an upper bound of the deterministic error, known as Koksma-Hlawka error bound \cite{QMCerror1992}, is
\begin{equation}
  \label{equ:KH bound}
  \abs{Q_N(u) - I(u)} \leq D(P_N)V(u).
\end{equation}

Here variation $V(u)$ is defined as
\begin{equation*}
  V(u) = \sum_{k=1}^{K}\sum_{1\leq i_1<\cdots<i_k\leq K}V_{Vit}^{(k)}(u; i_1,\cdots, i_k),
\end{equation*}
where $V_{Vit}^{(k)}(u; i_1, \cdots, i_k)$ is the variation in sense of Vitali applied to the restriction of $u$ to the space
of dimension $k\;\{(u_1, \cdots, u_K)\in \Omega : u_j = 1 \;\textrm{for}\; j\neq i_1, \cdots, i_k \}$. Precisely, let $\Delta(u, J)$
be the alternative sum of $u$ values at the edges of sub-interval $J$ when $P_N$ is a partition
of $\Omega=[0, 1]^K$, we give the definition of the variation in sense of Vitali as
\begin{equation*}
  V_{V_{it}}(u) = \sup_{P_N}\sum_{J\in P_N}\abs{\Delta(u, J)}.
\end{equation*}
$D(P_N)$ is defined to measure the discrepancy of the set $P_N$ as
\begin{equation*}
  D(P_N) = \sup_{x\in[0,1]^K}\abs{\frac{\sum_{n=0}^{N-1}1_B(x^n)}{N} - \prod_{i=1}^{K}x_i},
\end{equation*}
where $x=(x_1,x_2,\cdots,x_K)$. Note that the error bound in \eqref{equ:KH bound} is controlled by this discrepancy and $\lim\limits_{N\to+\infty}D(P_N) = 0$ if and only if the sequence $P_N$ is equi-distributed. A sequence $P_N$ is said to be a low-discrepancy sequence if $D(P_N) = O((\ln N)^K/N)$, the best-known result for infinite sequences. Therefore, QMC method converges much faster than MC method. Practically, the commonly used Sobol sequence is one of the low-discrepancy sequences \cite{Sobol1976,Bruno2004survey}.

\subsection{Stochastic gradient descent method}

In deep Ritz method, we use the SGD method to find the optimal set of parameters. The SGD method finds the optimal solution in an iterative way with the $i$-th iteration of the form
\begin{equation}
  \theta_{i+1} = \theta_i + \alpha_ig(\theta_i, \xi^k),
\end{equation}
where $g(\theta_i, \xi^k)$ is a stochastic vector
\begin{equation}
  g(\theta_i, \xi^k) = \frac 1{N}\sum_{k=1}^{N}\nabla I(\theta_i, \xi^k)
\end{equation}
obtained by a sampling strategy, $\xi^k$ is a sampling point, and $\alpha_i$ is the stepsize. Practically, we use ADAM \cite{ADAM} to accelerate the training process for both MC sampling and QMC sampling.

\section{Numerical results}\label{sec:numerical results}

Now we are ready to apply QMC sampling strategy to train the network structure of deep Ritz method in \secref{sec:QMC for deep-ritz sub:trail function}. For Dirichlet problems and some special Neumann problems, we do not need penalty terms on the boundary. However, for general Neumann problems, a penalty term must be added to the loss function and thus we have to approximate this term by sampling on the boundary. No matter in which case, numerically QMC method always performs better than MC method. We use the relative $L_2$ error for quantitative comparison in all examples
\begin{equation}\label{equ:L2 error}
  \begin{aligned}
  \mathrm{error} = \sqrt{\dfrac{\int_{\Omega}\left(\hat{u}_{\theta}(x)-u(x)\right)^2 \mathrm{d}x}{\int_{\Omega}u(x)^2\mathrm{d}x}}.
  \end{aligned}
\end{equation}

\subsection{Dirichlet problem}
Consider the Poisson equation over $\Omega=[-1,1]^K$
\begin{equation}\label{equ:poisson equation with DBC}
\left\{
  \begin{aligned}
  -\triangle u &= \pi^2 \sum_{k=1}^{K}\cos(\pi x_k) & x\in\Omega,\\
  u(x) &=\sum_{k=1}^{K}\cos(\pi x_k) & x\in \partial \Omega.
  \end{aligned}
  \right.
\end{equation}
The exact solution is $  u(x) =\sum_{k=1}^{K}\cos(\pi x_k)$. We construct the network in the form of
\begin{equation}\label{equ:exm1 network}
\left\{
\begin{aligned}
  &\hat{u}_{\theta}(x)= A(x)\cdot f_{\theta}(x) +B(x)& x\in \Omega,\\
  &A(x)=\exp \left(\prod_{k=1}^{K}(x_k^2-1)\right)-1 & x\in \Omega,\\
  &B(x)=\exp \left(\prod_{k=1}^{K}(x_k^2-1)\right)\sum_{k=1}^{K}\cos(\pi x_k) & x \in \Omega.
\end{aligned}
\right.
\end{equation}
It is easy to verify that $A(x)=0, B(x)=\sum_{k=1}^{K}\cos(\pi x_k)$ on the boundary, satisfying the structure defined in \eqref{equ:build for DBC}. Figure \ref{fig:compare true with train DBC} plots exact and trained solutions to \eqref{equ:poisson equation with DBC} in 2D and qualitative agreement is observed.
\begin{figure}
     \centering
     \begin{subfigure}[b]{0.48\textwidth}
         \centering
         \includegraphics[width=\textwidth]{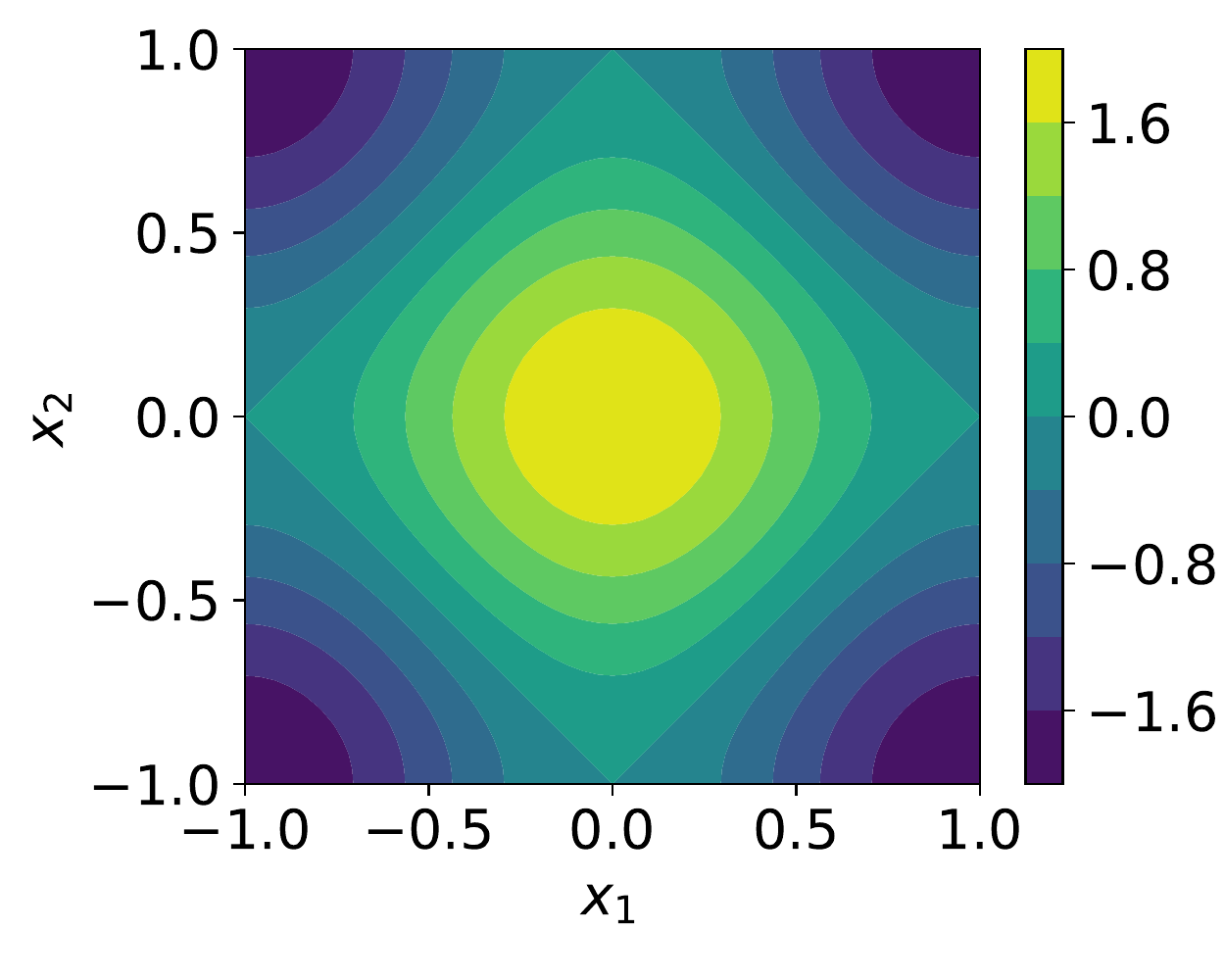}
         \caption{Exact solution}
         \label{fig:train_solution_DBC}
     \end{subfigure}
     \hfill
     \begin{subfigure}[b]{0.48\textwidth}
         \centering
         \includegraphics[width=\textwidth]{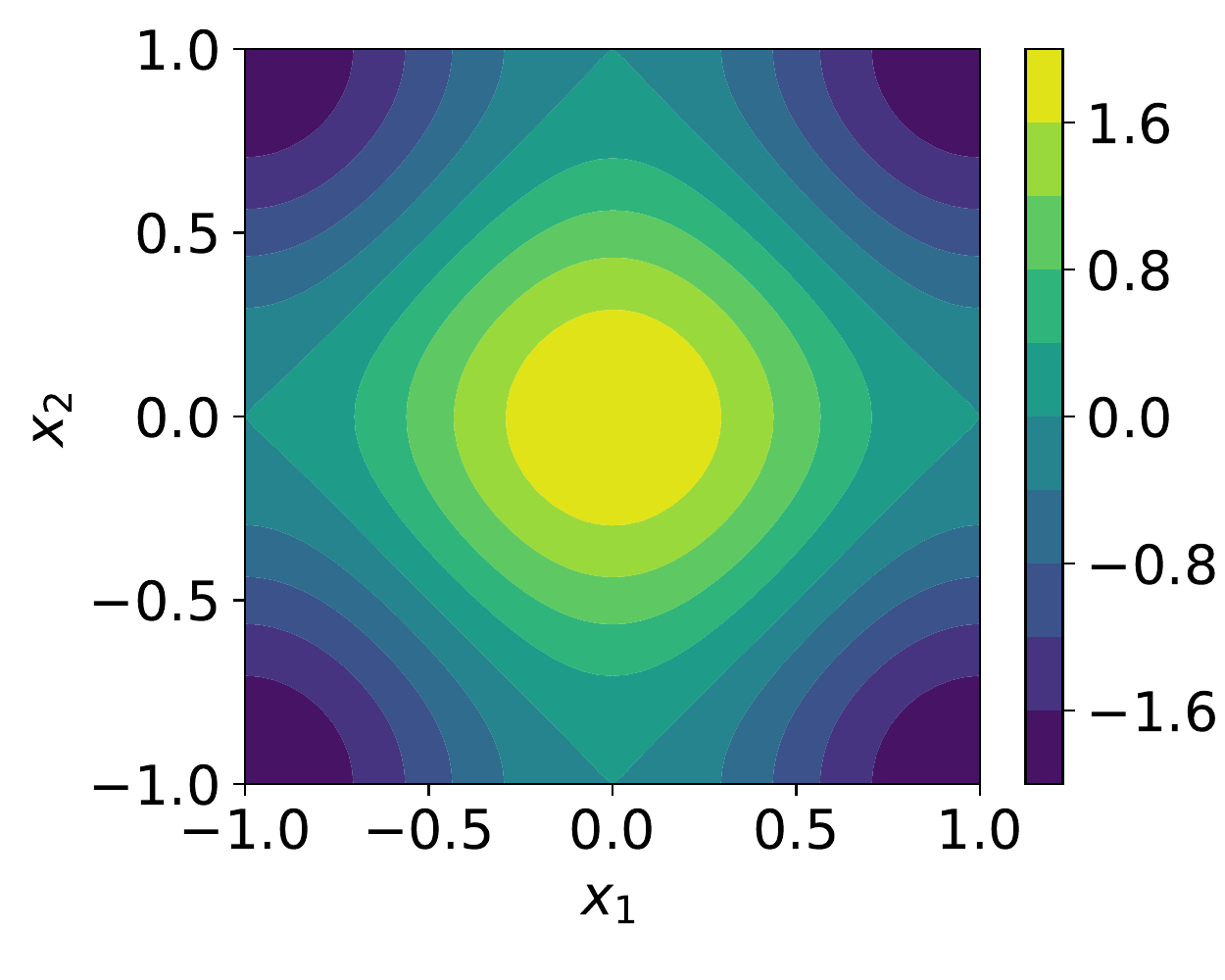}
         \caption{Trained solution}
         \label{fig:true_solution_DBC}
     \end{subfigure}
        \caption{Exact and trained solutions to problem \eqref{equ:poisson equation with DBC} in 2D. The exact solution $u(x) =\sum_{k=1}^{2}\cos(\pi x_k)$ and the approximate solution is trained with 5000 points (Sobol sequence) used at each iteration.}
        \label{fig:compare true with train DBC}
\end{figure}
Detailed setup of the neural network used for Dirichlet problem in different dimensions is recorded in Table \ref{tab:detail of network DBC}.
\begin{table}[htbp]
\centering
\caption{Detailed setup of the neural network used for Dirichlet problem in different dimensions, where $m$ denotes the number of nodes contained in each layer.}
\label{tab:detail of network DBC}
\begin{tabular}{llll}
\hline\noalign{\smallskip}
Dimension & Blocks Num& $m$ & Parameters  \\
\noalign{\smallskip}\hline\noalign{\smallskip}
2 & 3 & 8 & 465\\
4 & 4 & 16 &2274\\
8 & 4 & 20 &3561\\
16 & 4 &48 &19681\\
\noalign{\smallskip}\hline
\end{tabular}
\end{table}

Relative $L_2$ errors in different dimensions and the corresponding convergence rates with respect to the mini-batch size are shown in Table \ref{tab:error of DBC} and Table \ref{tab:convergence error of DBC}, respectively. Since there are some oscillations as the iteration increases, each point here represents the error averaged over $50$ iterations. The total number of iterations is set to be $10000$. It is reasonable to find that QMC method performs better than MC method as shown in Tables \ref{tab:error of DBC} and \ref{tab:convergence error of DBC}. When the size of mini-batches increases, the advantage of QMC method over MC method reduces. This is natural in the sense that both methods converges when the number of sampling points increases. We further plot detained training processes of QMC and MC methods in Figure \ref{fig:compare 2D DBC}. Clearly, QMC sampling reduces the magnitude of error of MC method by about three times on average with the same mini-batch size for the 2D problem.
\begin{table}[htbp]
	\centering
	\caption{Relative $L^2$ errors in different dimensions with different mini-batch sizes for Dirichlet problem and the onvergence order is recored with respect to $1/N$.}
	\label{tab:error of DBC}       
	\begin{tabular}{llllll}
		\hline\noalign{\smallskip}
		\multirow{2}*{Dimension} &\multirow{2}*{mini-batch size}&\multicolumn{2}{c}{QMC} &\multicolumn{2}{c}{MC}  \\
		~&~&\multicolumn{1}{c}{error($\times10^{-2}$)}& \multicolumn{1}{c}{order}&\multicolumn{1}{c}{error($\times10^{-2}$)}& \multicolumn{1}{c}{order}\\
		\noalign{\smallskip}\hline\noalign{\smallskip}
		\multirow{4}*{2D} & 500 & 1.7141& &4.2706&   \\
		~ & 1000 & 1.1420  & 0.59&3.4157 & 0.32\\
		~ & 2000 & 0.7702  & 0.59&2.6225 & 0.38\\
		~ & 4000 & 0.6401  & 0.27&2.2505 & 0.22\\
		\hline
		\multirow{4}*{4D} & 500 &
		1.8735 & &3.5183 &\\
		~ & 1000 & 1.4468 &0.37& 3.0786& 0.19 \\
		~ & 2000 & 1.0557 &0.45& 2.6561& 0.21\\
		~ & 4000 & 0.8076 &0.39& 2.0410& 0.38 \\
		\hline
		\multirow{4}*{8D} & 500 &
		2.0737 & & 2.4514 &\\
		~ & 1000 & 1.5714 &0.40& 2.2083& 0.15\\
		~ & 2000 & 1.1607 &0.43& 2.0297& 0.12\\
		~ & 10000 & 0.8139 &0.22 & 1.1551& 0.35\\
		\hline
		\multirow{4}*{16D}& 2000 & 0.8613 & &2.0754\\
		~ & 5000 & 0.6863 & 0.25 & 1.3506 &0.47\\
		~ & 10000 & 0.5361 &0.36& 1.1383 & 0.25\\
		~ & 20000 & 0.4623 &0.21& 1.2298 & 0.11
		\\
		\noalign{\smallskip}\hline
	\end{tabular}
\end{table}
\begin{table}[htbp]
\centering
\caption{Fitted convergence rates of the relative $L^2$ error with respect to the mini-batch size in different dimensions (recorded in terms of $1/N$) for Dirichlet problem.}
\label{tab:convergence error of DBC}       
\begin{tabular}{llll}
\hline\noalign{\smallskip}
Dimension & QMC & MC  \\
\noalign{\smallskip}\hline\noalign{\smallskip}
2D & 0.48 & 0.32\\
4D & 0.41 & 0.26\\
8D & 0.31 & 0.26\\
16D&0.28 & 0.24 \\
\noalign{\smallskip}\hline\noalign{\smallskip}
\end{tabular}
\end{table}
\begin{figure}
     \centering
     \begin{subfigure}[b]{0.48\textwidth}
         \centering
         \includegraphics[width=\textwidth]{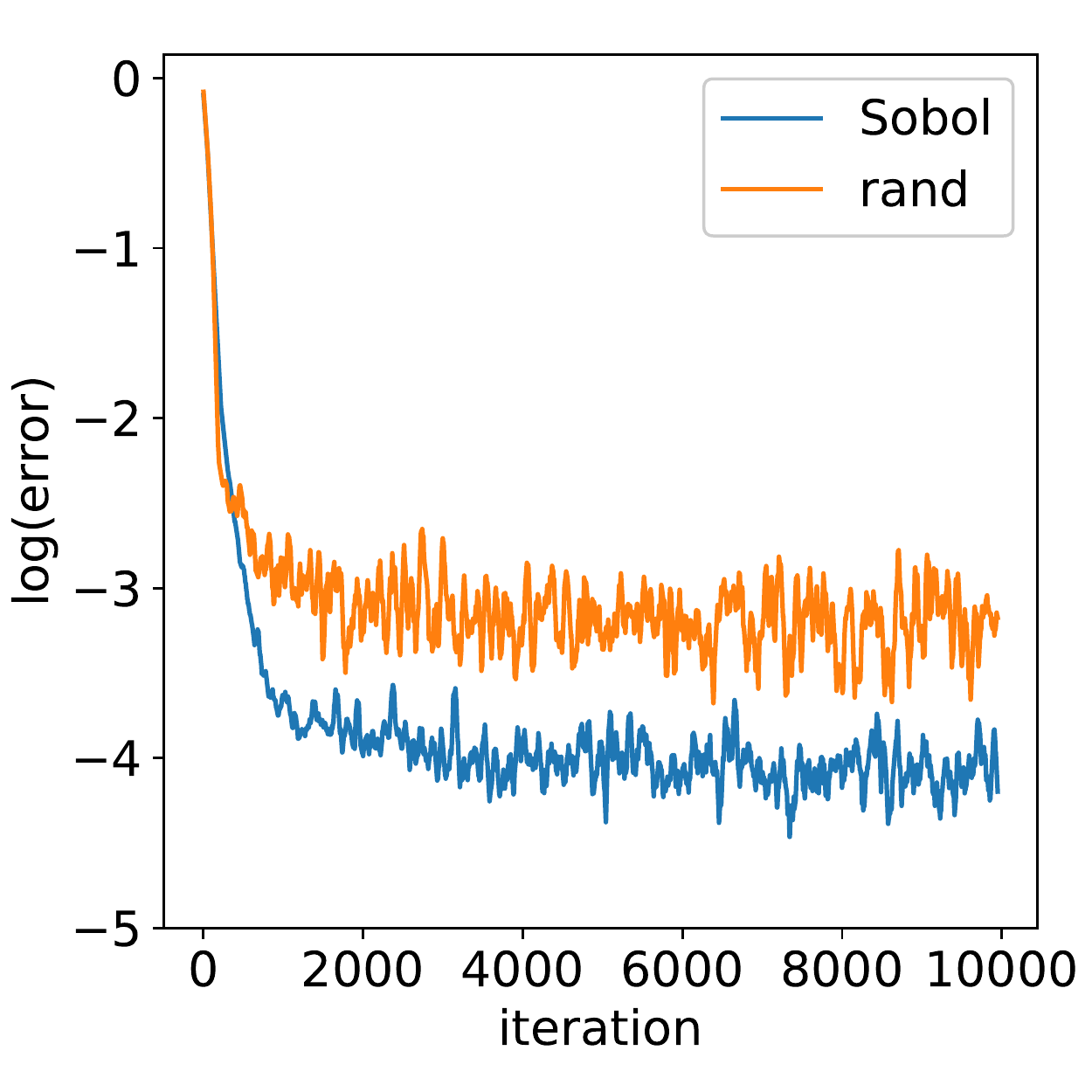}
         \caption{Mini-batch size: 500}
         \label{fig:2DcompareDBC500}
     \end{subfigure}
     \hfill
     \begin{subfigure}[b]{0.48\textwidth}
         \centering
         \includegraphics[width=\textwidth]{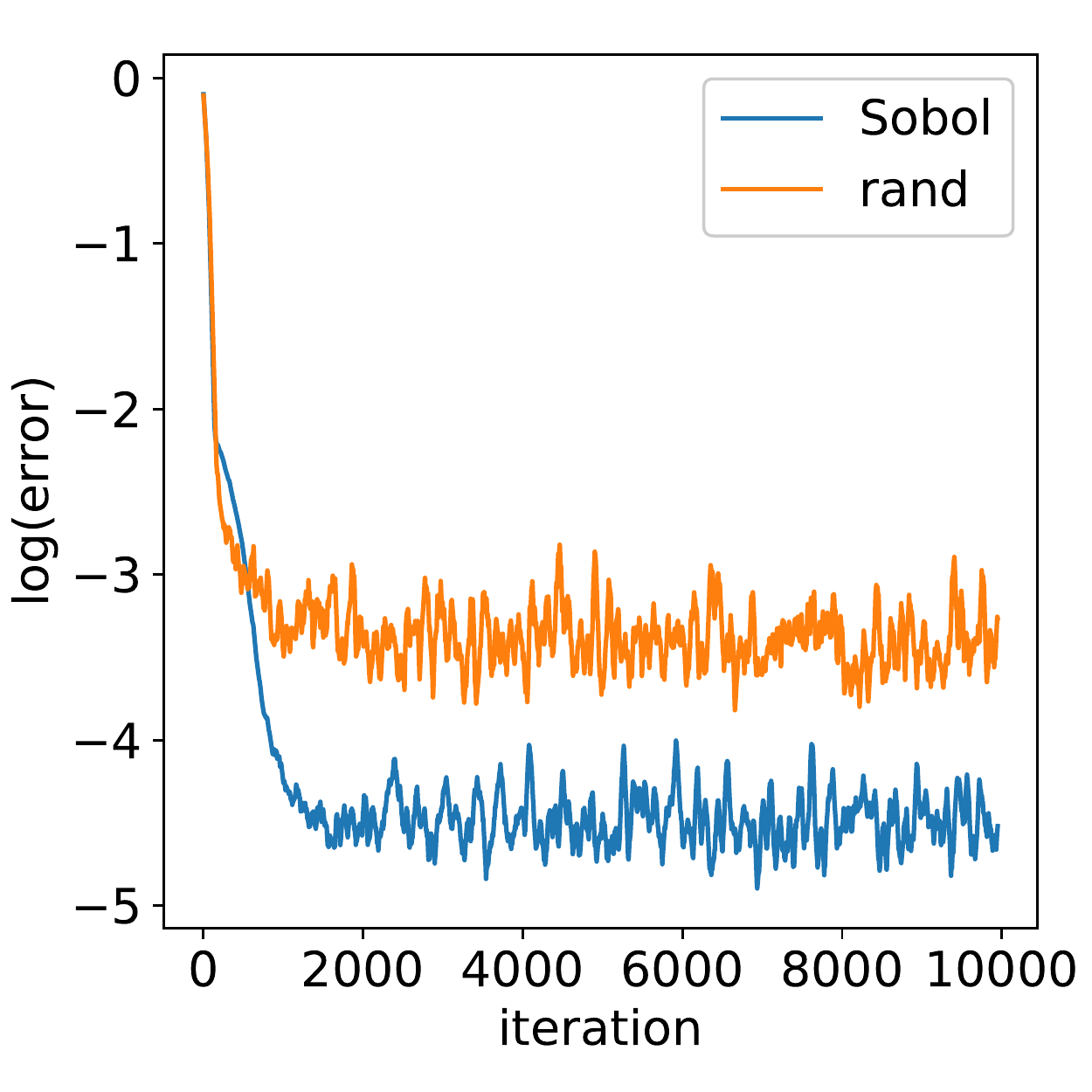}
         \caption{Mini-batch size: 1000}
         \label{fig:2DcompareDBC1000}
     \end{subfigure}
     \begin{subfigure}[b]{0.48\textwidth}
         \centering
         \includegraphics[width=\textwidth]{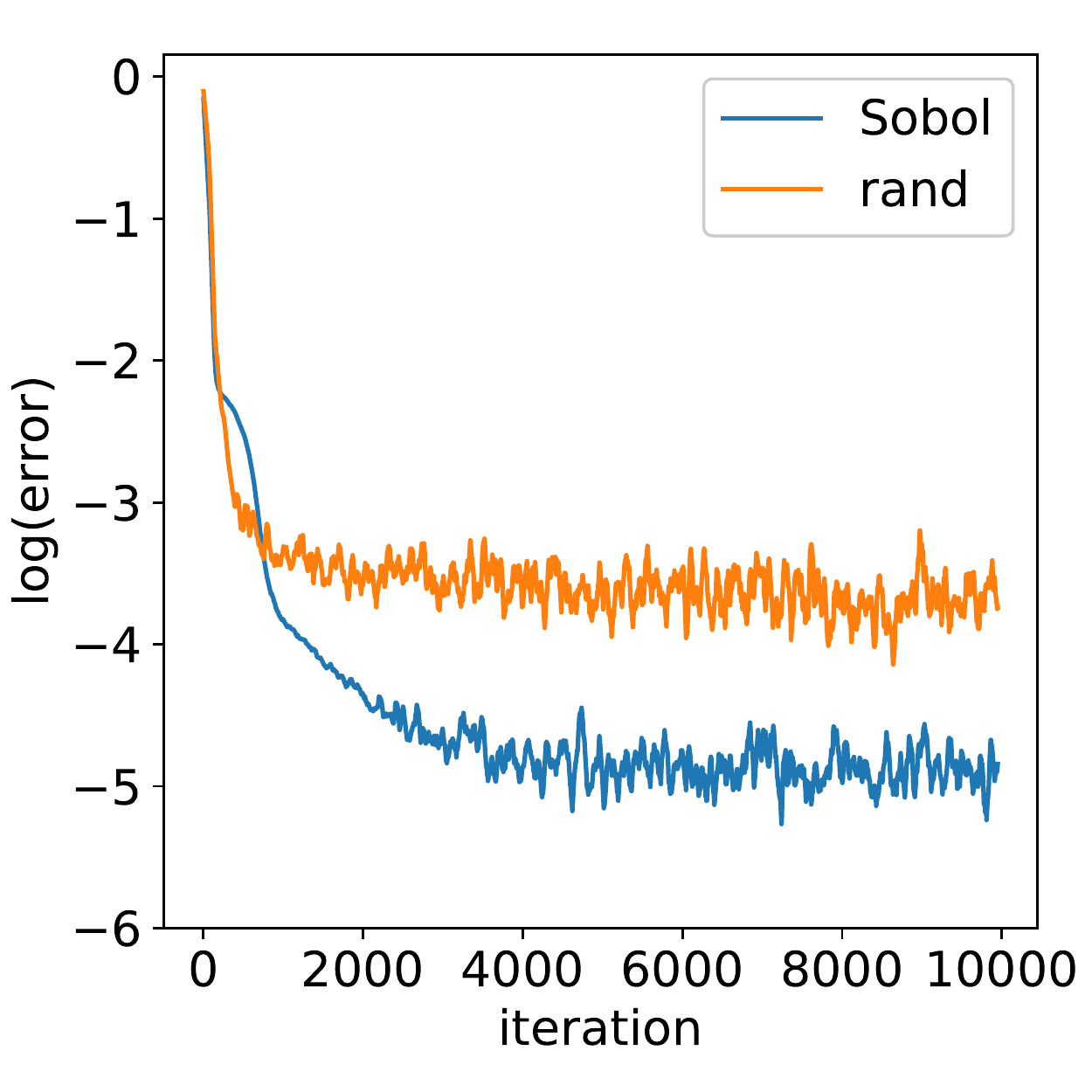}
         \caption{Mini-batch size: 2000}
         \label{fig:2DcompareDBC2000}
     \end{subfigure}
     \hfill
     \begin{subfigure}[b]{0.48\textwidth}
         \centering
         \includegraphics[width=\textwidth]{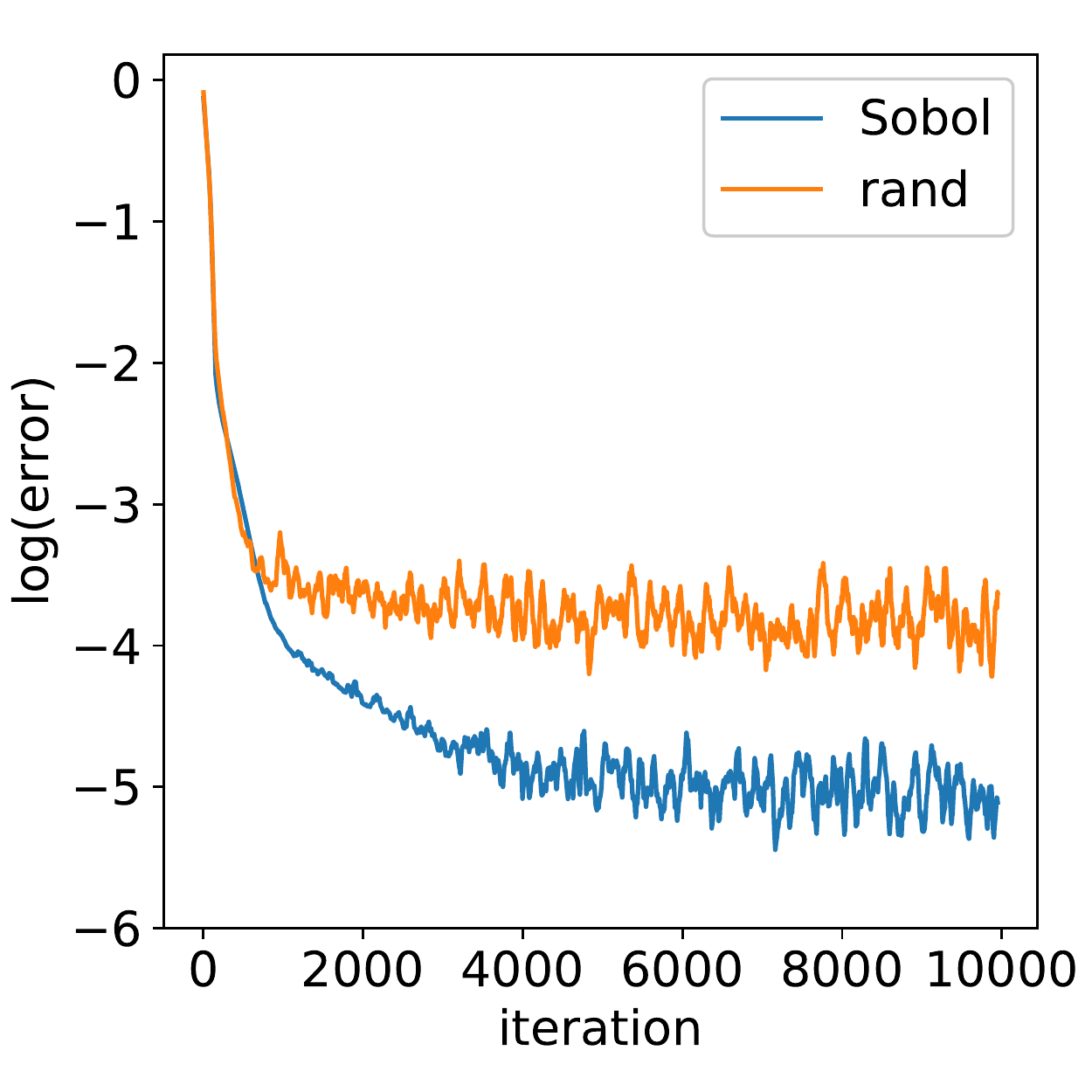}
         \caption{Mini-batch size: 4000}
         \label{fig:2DcompareDBC4000}
     \end{subfigure}
        \caption{Detained training processes of QMC and MC methods for different mini-batch sizes for Dirichlet problem in 2D. The log function here uses $e$ as base.}
        \label{fig:compare 2D DBC}
\end{figure}
Figure \ref{fig:comparebatchsizeonerrorDBC} plots relative $L^2$ error in terms of mini-batch size for QMC and MC methods for \eqref{equ:poisson equation with DBC} from 2D to 16D. A clear evidence is that QMC method always outperforms MC method and the error reduction is significant.
\begin{figure}
	\centering
	\begin{subfigure}[b]{0.48\textwidth}
		\centering
		\includegraphics[width=\textwidth]{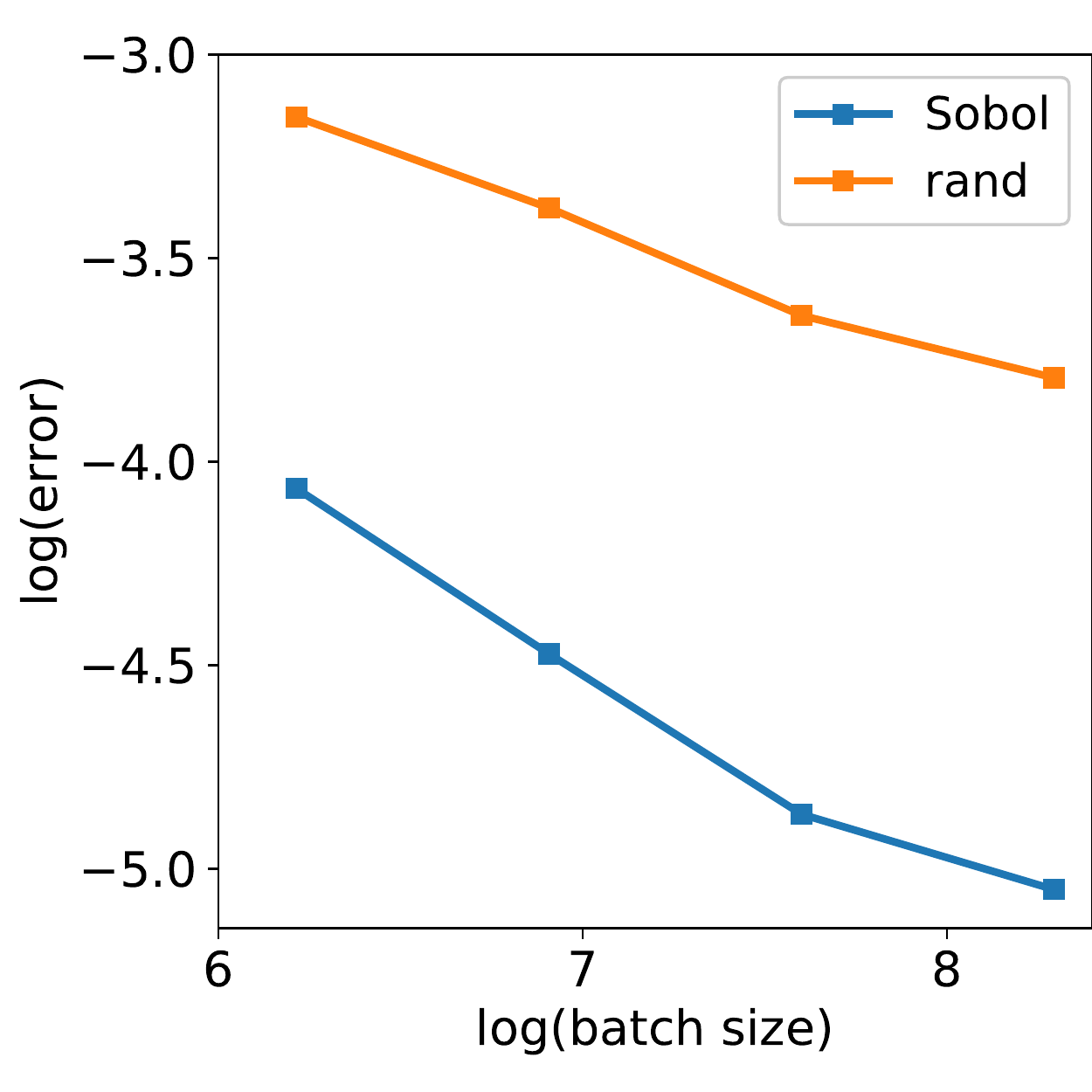}
		\caption{2D}
		\label{fig:2comparebatchsizeonerror}
	\end{subfigure}
	\hfill
	\begin{subfigure}[b]{0.48\textwidth}
		\centering
		\includegraphics[width=\textwidth]{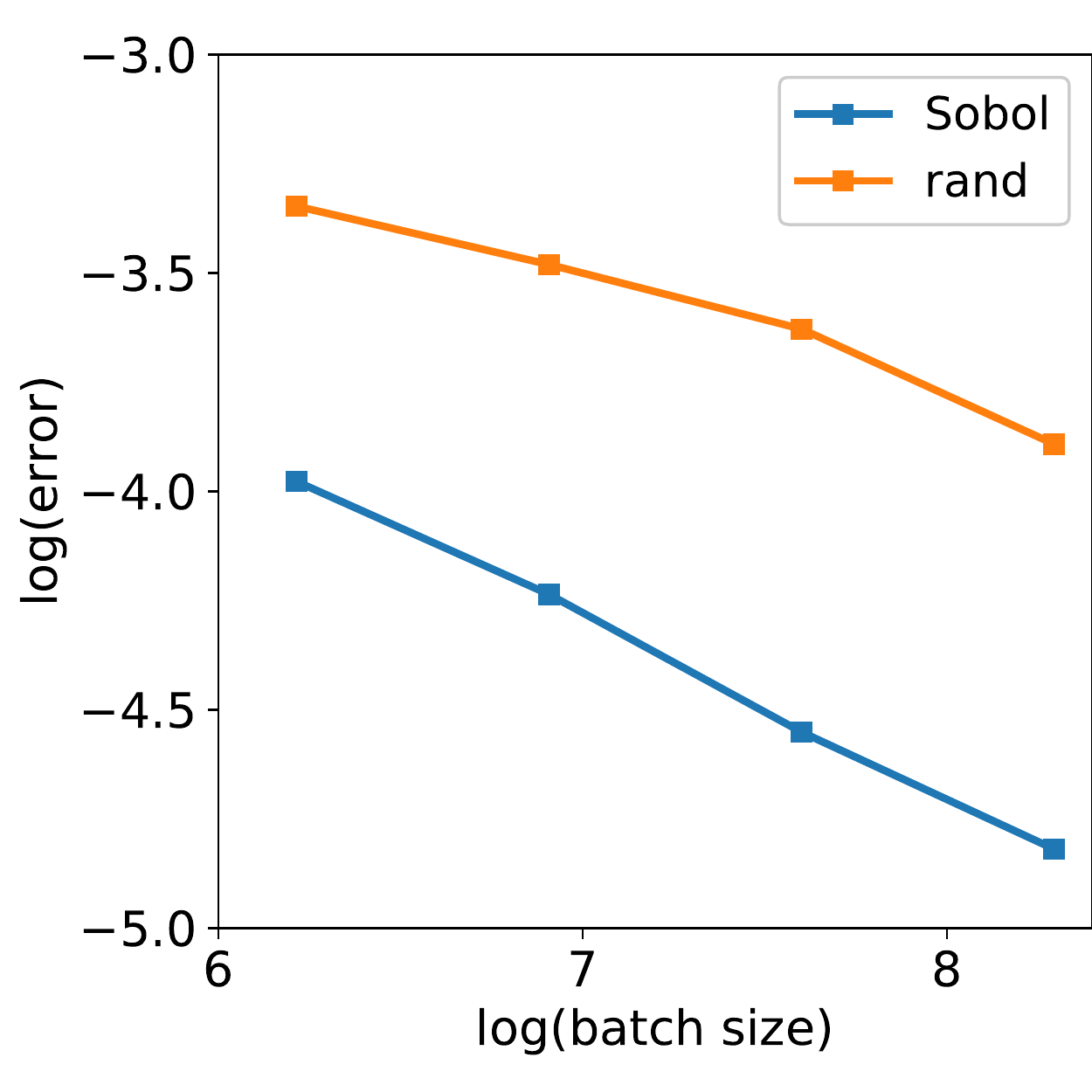}
		\caption{4D}
		\label{fig:4comparebatchsizeonerror}
	\end{subfigure}
	\begin{subfigure}[b]{0.48\textwidth}
		\centering
		\includegraphics[width=\textwidth]{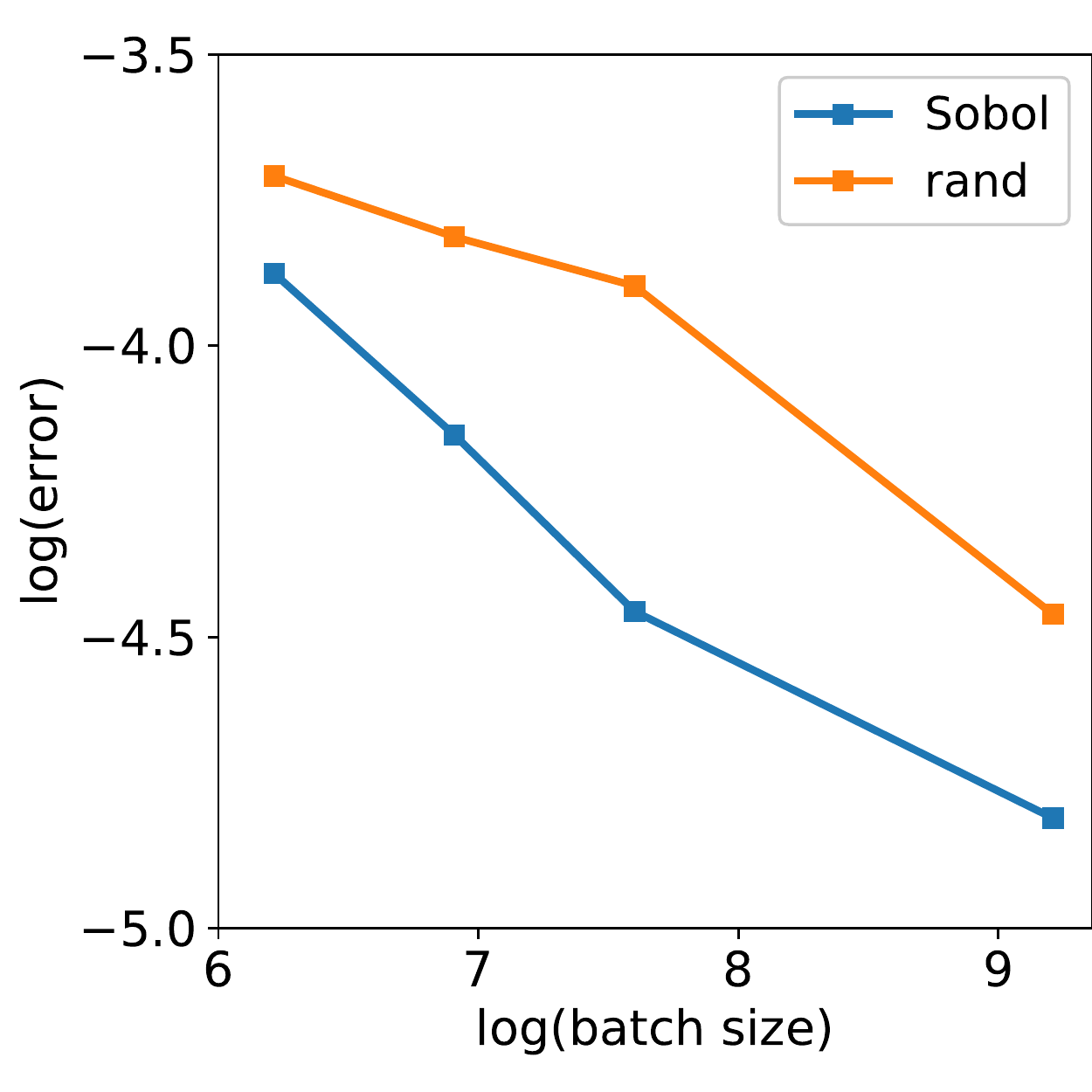}
		\caption{8D}
		\label{fig:8comparebatchsizeonerror}
	\end{subfigure}
	\hfill
	\begin{subfigure}[b]{0.48\textwidth}
		\centering
		\includegraphics[width=\textwidth]{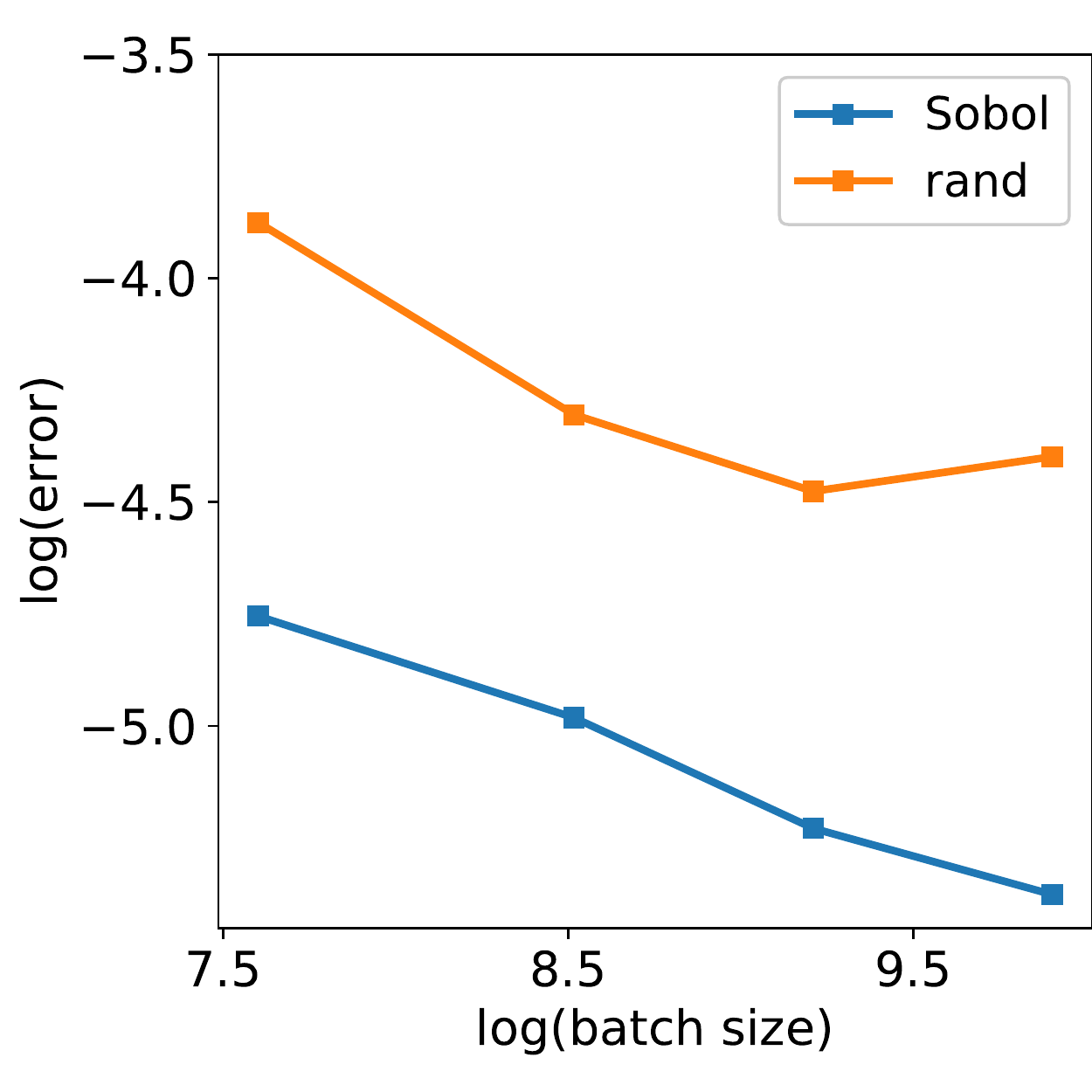}
		\caption{16D}
		\label{fig:16comparebatchsizeonerror}
	\end{subfigure}
	\caption{Relative $L^2$ error versus mini-batch size in QMC and MC methods for \eqref{equ:poisson equation with DBC} from 2D to 16D. The log function here uses $e$ as base.}
	\label{fig:comparebatchsizeonerrorDBC}
\end{figure}

From a different perspective, in order to achieve the same approximation accuracy, we may ask how much cheaper QMC method is compared to MC method. To do this, we fix the mini-batch size in MC method to be $10000$ and $100000$, and check the corresponding size in QMC method for the same accuracy requirement. Results are recorded in Table \ref{tab:comparison of MC and QMC of DBC}, and detailed training processes in 2D are plotted in Figure \ref{fig:compare 2D DBC same accuracy}.
For the same accuracy requirement, compared to MC method, QMC method reduces the size of training data set by more than two orders of magnitude.
\begin{table}[htbp]
\centering
\caption{Comparison of mini-batch sizes in QMC and MC methods for Dirichlet problem in different dimensions when the same approximation accuracy is required.}
\label{tab:comparison of MC and QMC of DBC}       
\begin{tabular}{lllll}
\hline\noalign{\smallskip}
\multirow{2}*{Dimension} &\multicolumn{2}{c}{MC} &\multicolumn{2}{c}{QMC}  \\
~&size& error($\times10^{-2}$)&size& error($\times10^{-2}$)\\
\noalign{\smallskip}\hline\noalign{\smallskip}
\multirow{2}*{2D} & 10000 & 1.3566 & 500 &1.7141 \\
~ & 100000 & 0.7702 & 2000 & 0.6149\\
\multirow{2}*{4D} & 10000 & 1.8462& 500 &1.8735   \\
~ & 100000 & 0.8986 &  4000 &0.8076 \\
\multirow{2}*{8D} & 10000 & 1.1551 & 2000 & 1.1607 \\
~ & 100000 & 0.8668 & 10000 &0.8139  \\
\multirow{2}*{16D} & 10000 & 1.1383 & 1000 & 1.0472 \\
~ & 100000 & 0.7993 & 5000 & 0.6863
 \\
\noalign{\smallskip}\hline
\end{tabular}
\end{table}

\begin{figure}
     \centering
     \begin{subfigure}[b]{0.48\textwidth}
         \centering
         \includegraphics[width=\textwidth]{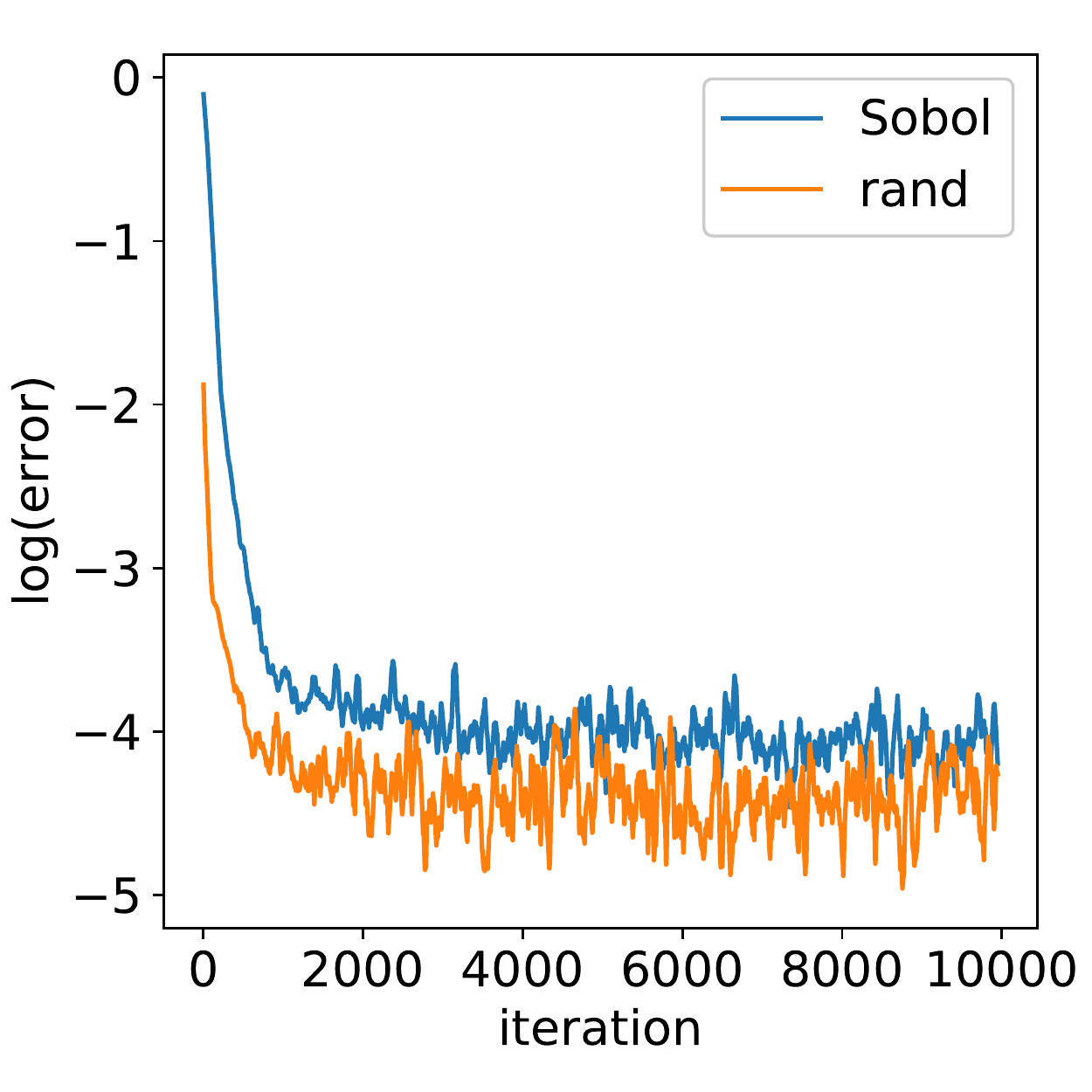}
         \caption{Mini-batch size: $10000$ (MC) and $500$ (QMC)}
         \label{fig:fig:compare 2D DBC same accuracy500}
     \end{subfigure}
     \hfill
     \begin{subfigure}[b]{0.48\textwidth}
         \centering
         \includegraphics[width=\textwidth]{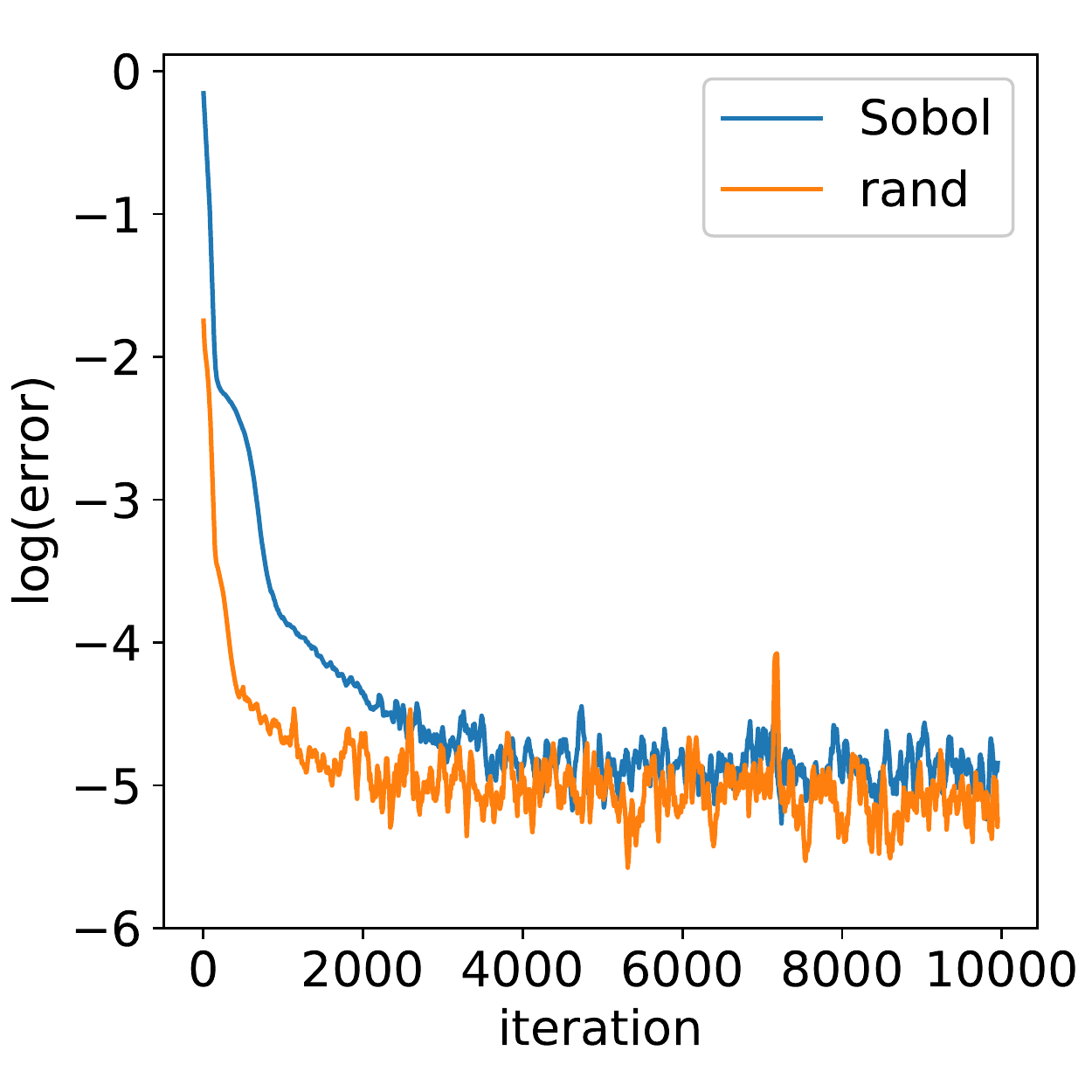}
         \caption{Mini-batch size: $100000$ (MC) and $2000$ (QMC)}
         \label{fig:compare 2D DBC same accuracy2000}
     \end{subfigure}
        \caption{Detailed training processes in QMC and MC methods for Dirichlet problem in 2D with different mini-batch sizes for the same accuracy requirement. The log function here uses $e$ as base.}
        \label{fig:compare 2D DBC same accuracy}
\end{figure}

All above results show that QMC method always performs better than MC method, typically by several times in terms of accuracy when the same mini-batch size is enforced or by more than two orders of magnitude in terms of efficiency when the same accuracy is required. However, their convergence rates with respect to both the iteration number and the size of training data set seem to be the same, due to the ADAM optimizer we used and the nonconvex nature of loss function.

\subsection{Neumann problem}
\label{sec:NBC}

Consider a Neumann problem over $\Omega=[0,1]^K$
\begin{equation}\label{equ:NBC problem}
\left\{
\begin{aligned}
  &-\Delta u + \pi^2 u = 2 \pi^2 \sum_{k}^{K} \cos (\pi x_k) & x\in \Omega\\
  &\left.\frac{\partial u}{\partial n} \right|_{\partial [0,1]^K}=0 & x\in\partial \Omega
  \end{aligned}
   \right.
\end{equation}
with the exact solution $u(x)=\sum_{k}\cos(\pi x_k)$ and $n$ being the unit outward normal vector. For this problem, we still do not need to add any penalty term for the boundary condition and the loss function is defined as
\begin{equation}\label{equ:loss for NBC}
  I(u)=\int_{\Omega}\left(\frac{1}{2}\left|\nabla u(x)\right|^2 +\pi^2 u(x)^2 -f(x)u(x) \right) \mathrm{d}x.
\end{equation}
The network structure is the same as before; see \eqref{equ:approximate u}. Detailed setup is listed in Table \ref{tab:detail of network NBC}. Exact and trained solutions in 2D are visualized in Figure \ref{fig:compare true with train NBC}. Relative $L_2$ errors in differential dimensions are recorded in Table \ref{tab:error of NBC} and the corresponding convergence rates are shown in Table \ref{tab:convergence error of NBC}.
\begin{table}[htbp]
\centering
\caption{Detailed setup of the neural network used for Neumann problem in different dimensions, where $m$ denotes the number of nodes contained in each layer.}
\label{tab:detail of network NBC}       
\begin{tabular}{llll}
\hline\noalign{\smallskip}
Dimension & Blocks & $m$ & Parameters  \\
\noalign{\smallskip}\hline\noalign{\smallskip}
2 & 4 & 10 & 921\\
4 & 4 & 15 &2011\\
8 & 4 & 30 &7741\\
16 & 5 &48 &24385\\
\noalign{\smallskip}\hline
\end{tabular}
\end{table}

\begin{figure}
     \centering
     \begin{subfigure}[b]{0.48\textwidth}
         \centering
         \includegraphics[width=\textwidth]{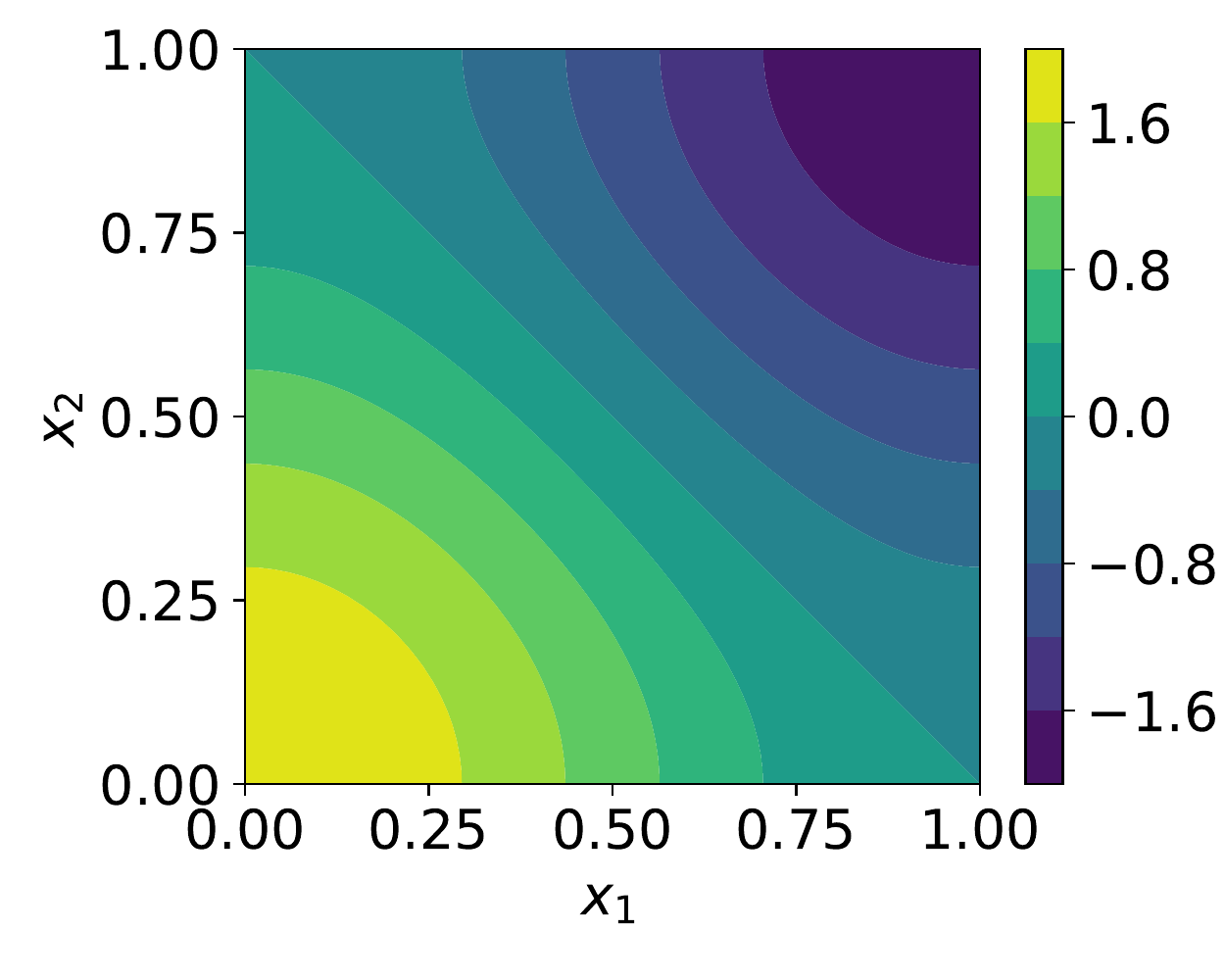}
         \caption{Exact solution}
         \label{fig:train_solution_NBC}
     \end{subfigure}
     \hfill
     \begin{subfigure}[b]{0.48\textwidth}
         \centering
         \includegraphics[width=\textwidth]{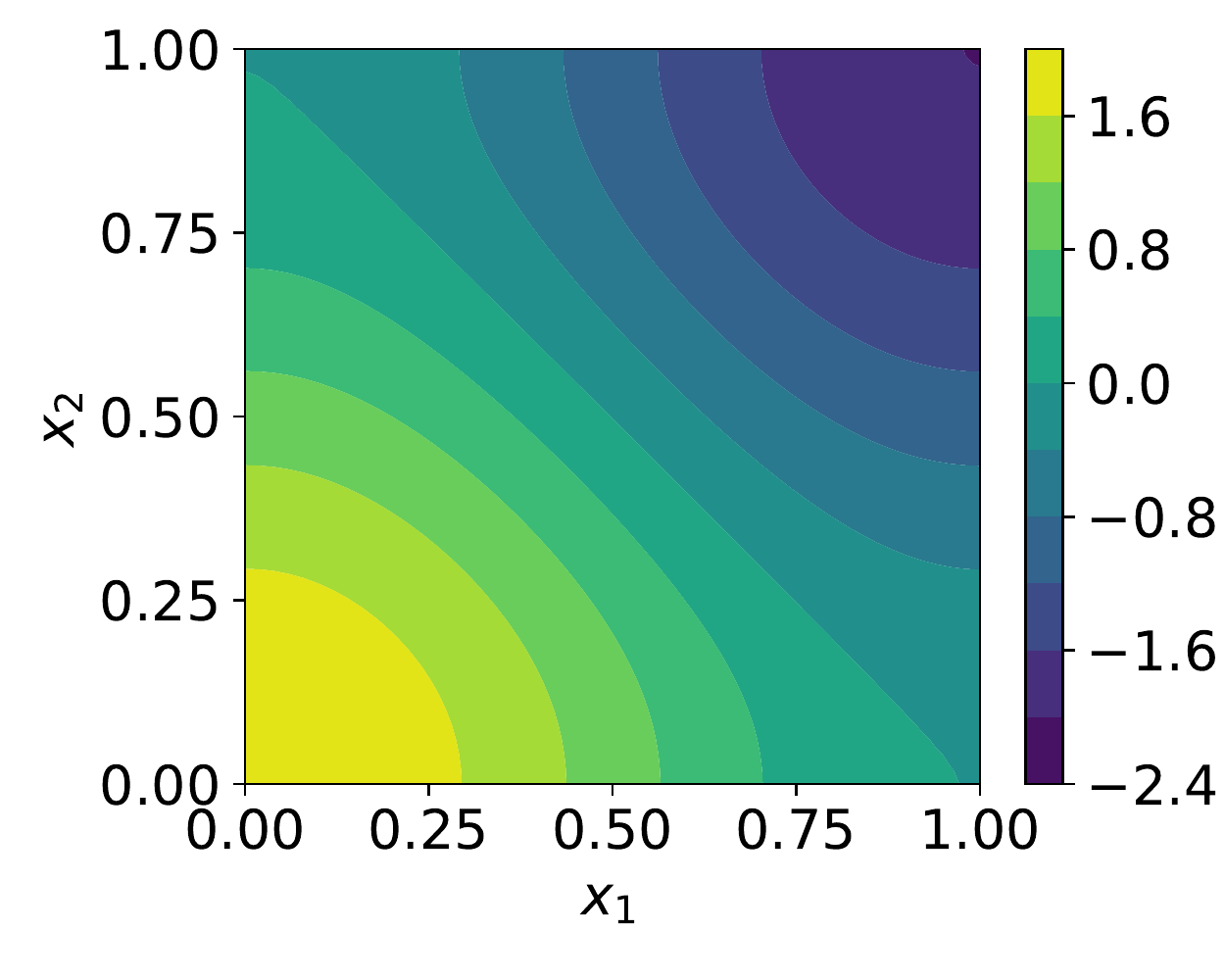}
         \caption{Trained solution}
         \label{fig:true_solution_NBC}
     \end{subfigure}
        \caption{Exact and trained solutions to problem \eqref{equ:NBC problem} in 2D. The exact solution $u(x)=\sum_{k=1}^2\cos(\pi x_k)$ and the approximate solution is trained with 5000 points (Sobol sequence) used at each iteration.}
        \label{fig:compare true with train NBC}
\end{figure}

\begin{table}[htbp]
\centering
\caption{Relative $L^2$ errors in different dimensions with different mini-batch sizes for Neumann problem and the onvergence order is recored with respect to $1/N$.}
\label{tab:error of NBC}       
\begin{tabular}{llllll}
\hline\noalign{\smallskip}
\multirow{2}*{Dimension} &\multirow{2}*{mini-batch size}&\multicolumn{2}{c}{QMC} &\multicolumn{2}{c}{MC}  \\
~&~&\multicolumn{1}{c}{error($\times10^{-2}$)}& \multicolumn{1}{c}{order}&\multicolumn{1}{c}{error($\times10^{-2}$)}& \multicolumn{1}{c}{order}\\
\noalign{\smallskip}\hline\noalign{\smallskip}
\multirow{4}*{2D}
    & 250 & 1.1218 & & 2.9571\\
~ & 500 & 0.7484   & 0.58 & 2.6754 & 0.14  \\
~ & 1000 & 0.4314 & 0.79 & 2.2572 & 0.25\\
~ & 2000 & 0.3140 & 0.46 & 1.9914 & 0.18\\
\hline
\multirow{4}*{4D} & 500 &
1.4219 & &3.9305\\
~ & 1000 & 0.9185 & 0.63 & 3.6990& 0.09\\
~ & 2000 & 0.3289 & 1.48 & 2.4779& 0.58\\
~ & 4000 & 0.2649 & 0.24  & 2.0152 & 0.23 \\
\hline
\multirow{4}*{8D} & 500 &
3.8542& &7.7874\\
~ & 1000 & 2.7353 & 0.49 &6.2379 &0.32\\
~ & 2000 & 2.4249 &  0.17  &5.4659 &0.19 \\
~ & 10000 & 1.8235 &  0.18 &2.6860 &0.44 \\
\hline
\multirow{4}*{16D} & 1000 & 3.2668 & &
6.0573\\
~&2000& 2.9308 &  0.16 &5.8566 &0.05\\
~&5000 & 2.9205 & 0.01 &4.8035 &0.22\\
~&10000& 1.5636 &  0.90 &3.8518 &0.32\\
\noalign{\smallskip}\hline
\end{tabular}
\end{table}
\begin{table}[htbp]
	\centering
	\caption{Fitted convergence rates of the relative $L^2$ error with respect to the mini-batch size in different dimensions (recorded in terms of $1/N$) for Neumann problem.}
	\label{tab:convergence error of NBC}       
	\begin{tabular}{llll}
		\hline\noalign{\smallskip}
		Dimension & QMC & MC  \\
		\noalign{\smallskip}\hline\noalign{\smallskip}
		2D & 0.63 & 0.20\\
		4D & 0.78 & 0.32\\
		8D & 0.23 & 0.35\\
		16D&0.28 & 0.20 \\
		\noalign{\smallskip}\hline\noalign{\smallskip}
	\end{tabular}
\end{table}

Figure \ref{fig:compare true with train NBC} plots detailed training processes of QMC and MC methods for different mini-batch sizes in 2D. Similarly, QMC method outperforms MC method by several times in terms of accuracy for the same size of training data set.
\begin{figure}
     \centering
     \begin{subfigure}[b]{0.48\textwidth}
         \centering
         \includegraphics[width=\textwidth]{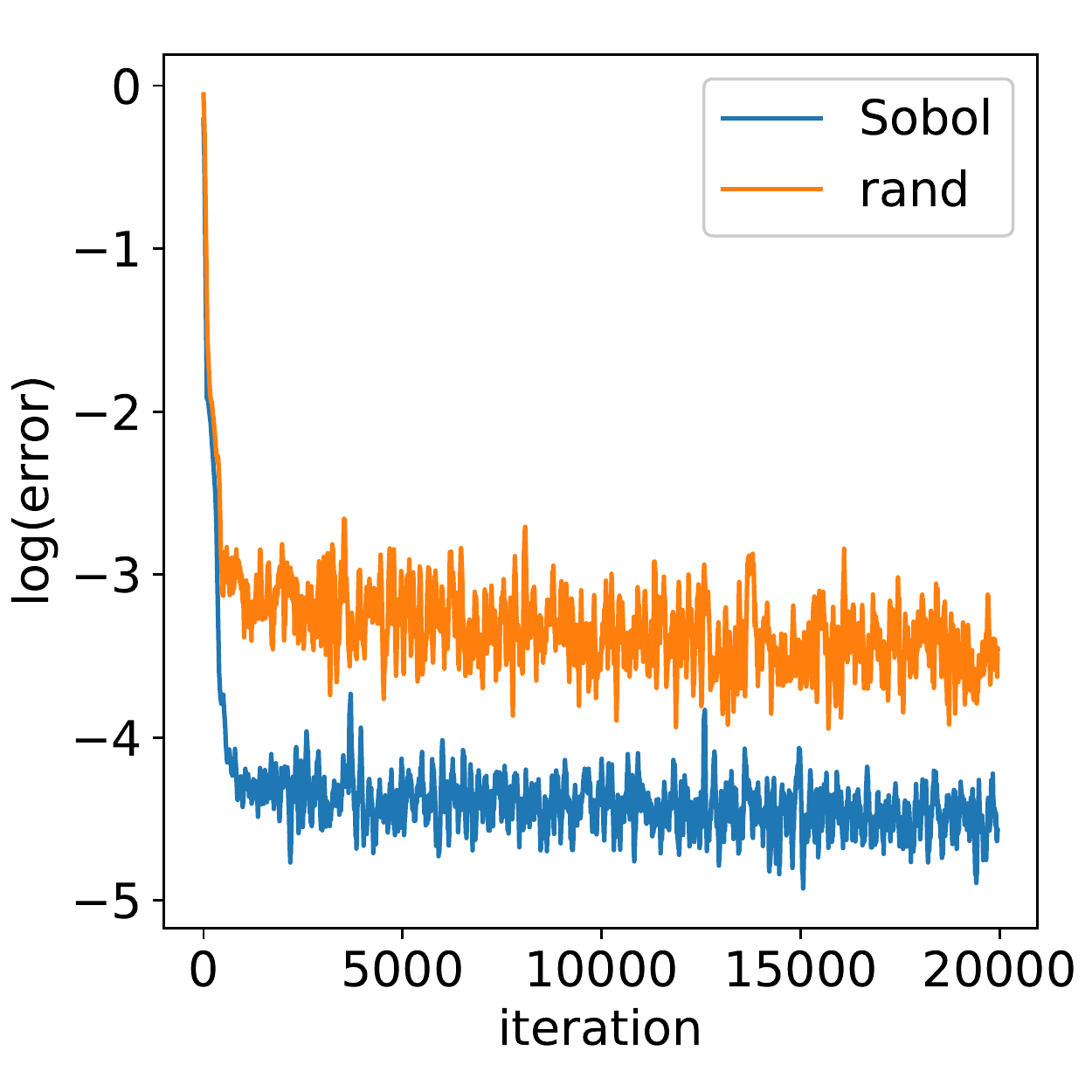}
         \caption{Mini-batch size: $250$}
         \label{fig:2DcompareNBC250}
     \end{subfigure}
     \hfill
     \begin{subfigure}[b]{0.48\textwidth}
         \centering
         \includegraphics[width=\textwidth]{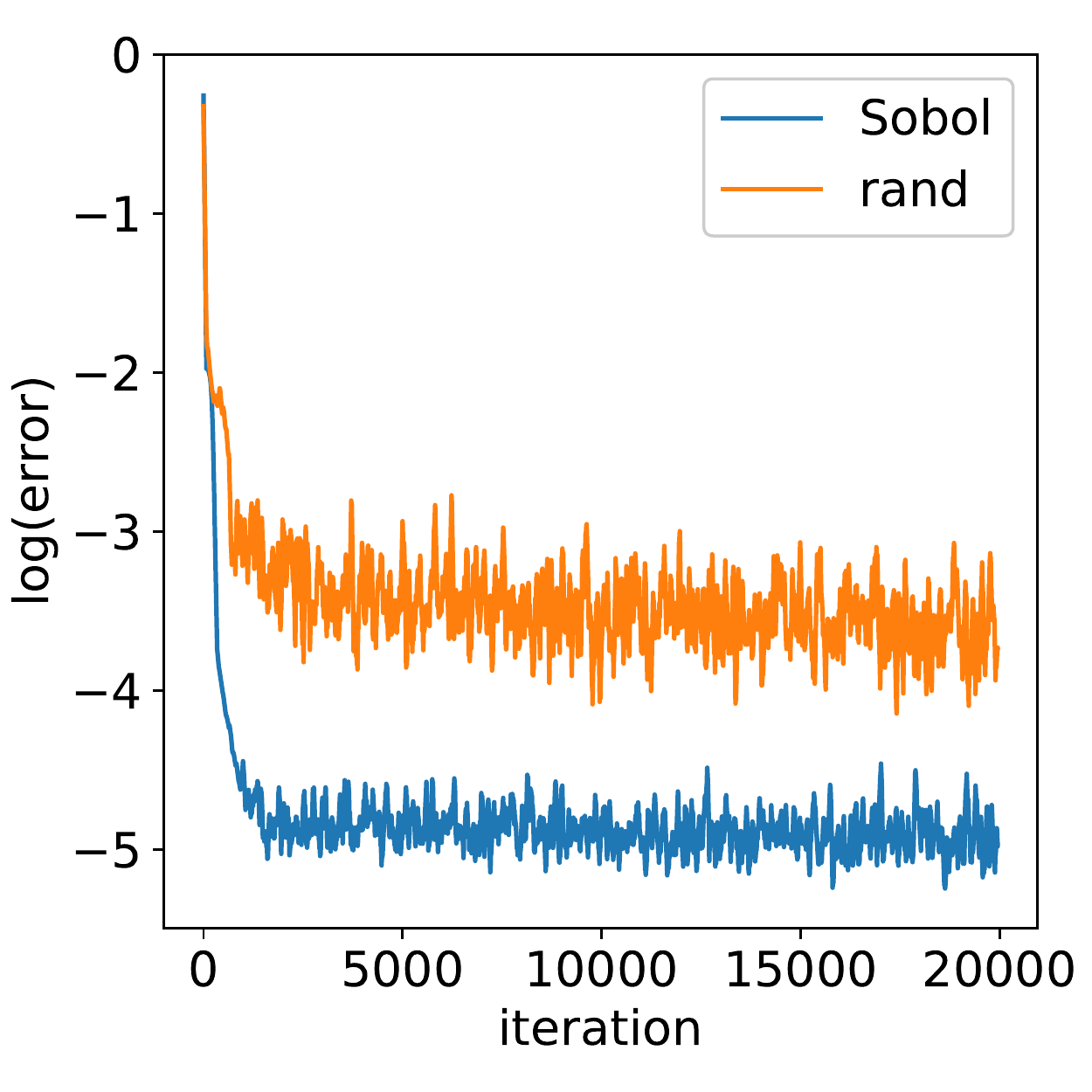}
         \caption{Mini-batch size: $500$}
         \label{fig:2DcompareNBC500}
     \end{subfigure}
          \hfill
     \begin{subfigure}[b]{0.48\textwidth}
         \centering
         \includegraphics[width=\textwidth]{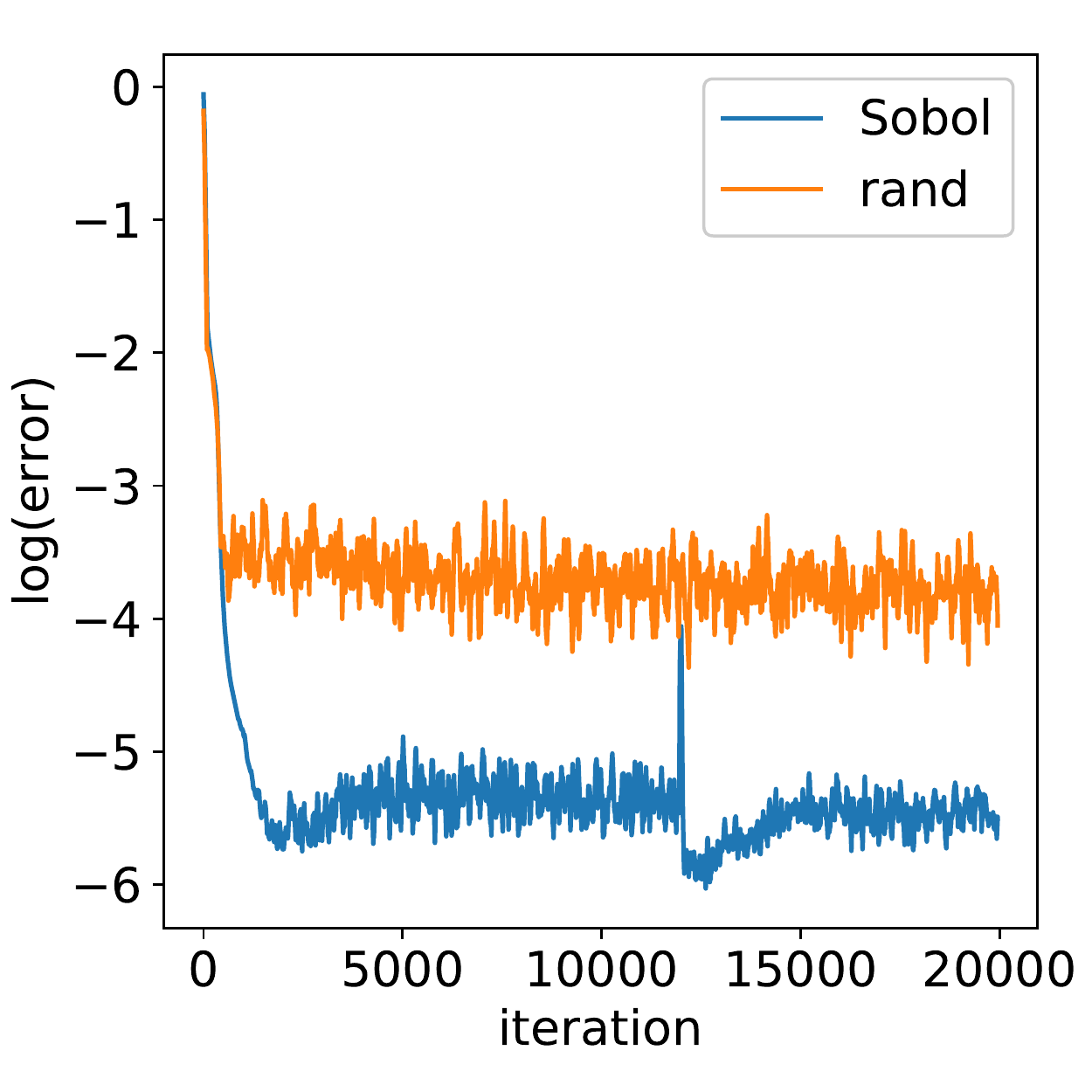}
         \caption{Mini-batch size: $1000$}
         \label{fig:2DcompareNBC1000}
     \end{subfigure}
          \hfill
     \begin{subfigure}[b]{0.48\textwidth}
         \centering
         \includegraphics[width=\textwidth]{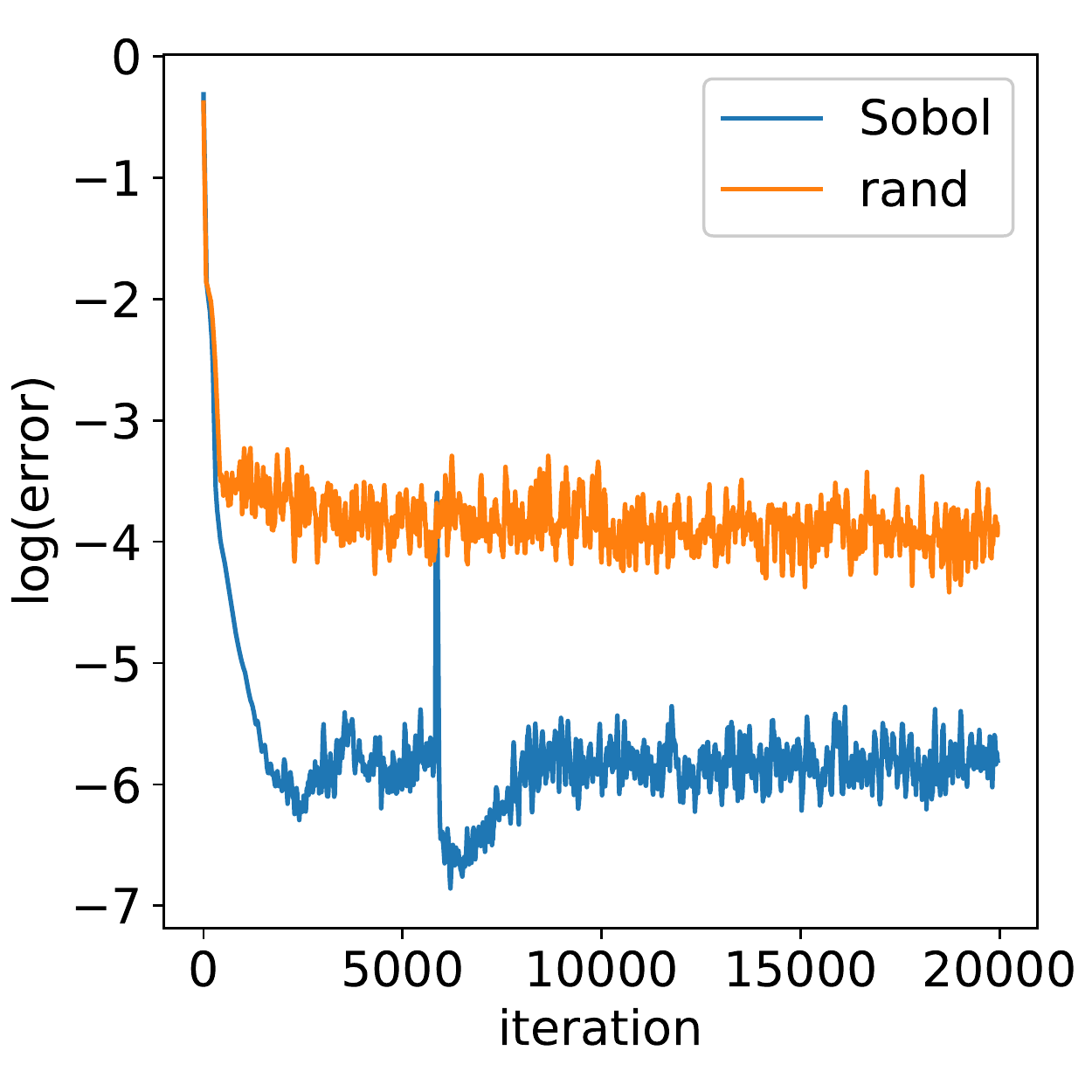}
         \caption{Mini-batch size: $2000$}
         \label{fig:2DcompareNBC2000}
     \end{subfigure}
        \caption{Detailed training processes of QMC and MC methods for different mini-batch sizes for Neumann problem in 2D. The log function here uses $e$ as base.}
        \label{fig:compare 2D NBC}
\end{figure}
Besides, for the same accuracy requirement, QMC reduces the size of training data set by orders of magnitudes; see Table \ref{tab:comparison of MC and QMC of NBC} as well as Figure \ref{fig:compare 2D NBC same accuracy} for details.
\begin{figure}
     \centering
     \begin{subfigure}[b]{0.48\textwidth}
         \centering
         \includegraphics[width=\textwidth]{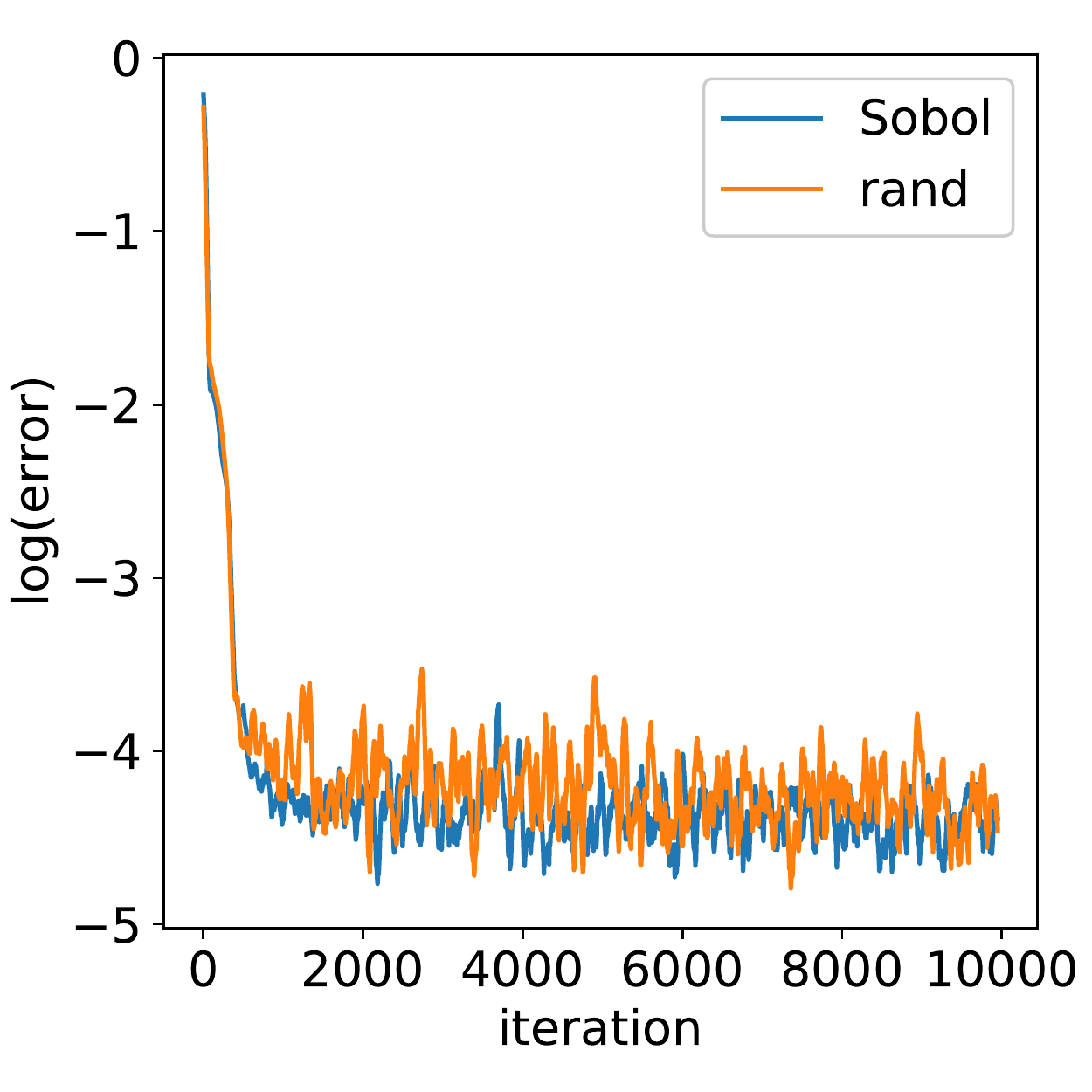}
         \caption{Mini-batch size: $10000$ (MC) and $250$ (QMC)}
         \label{fig:fig:compare 2D NBC same accuracy500}
     \end{subfigure}
     \hfill
     \begin{subfigure}[b]{0.48\textwidth}
         \centering
         \includegraphics[width=\textwidth]{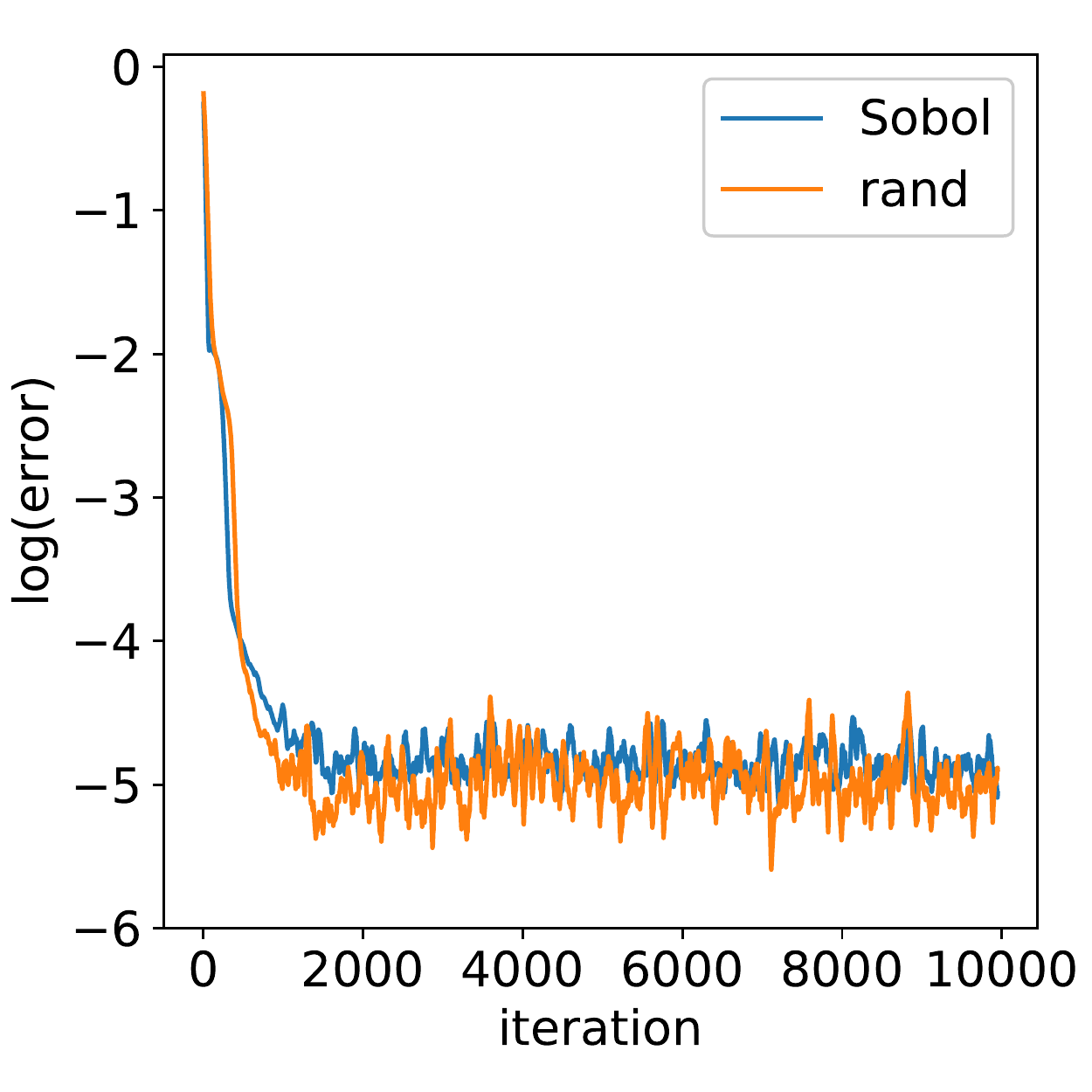}
         \caption{Mini-batch size: $100000$ (MC) and $500$ (QMC)}
         \label{fig:compare 2D NBC same accuracy500}
     \end{subfigure}
        \caption{Detailed training processes in QMC and MC methods for Neumann problem in 2D with different mini-batch sizes for the same accuracy requirement. The log function here uses $e$ as base.}
        \label{fig:compare 2D NBC same accuracy}
\end{figure}
\begin{table}[htbp]
\centering
\caption{Comparison of mini-batch sizes in QMC and MC methods for Neumann problem in different dimensions when the same approximation accuracy is required.}
\label{tab:comparison of MC and QMC of NBC}       
\begin{tabular}{lllll}
\hline\noalign{\smallskip}
\multirow{2}*{Dimension} &\multicolumn{2}{c}{MC} &\multicolumn{2}{c}{QMC}  \\
~&size& error($\times10^{-2}$)&size& error($\times10^{-2}$)\\
\noalign{\smallskip}\hline\noalign{\smallskip}
\multirow{2}*{2D} &10000 & 1.3138 & 250 & 1.1218 \\
~&100000 & 0.6536
& 500 &	0.7484\\
\multirow{2}*{4D}
&10000	& 2.0778
&500	& 1.4218 \\
~ & 100000 & 0.4622
&2000 &	0.3289 \\
\multirow{2}*{8D}
&10000	& 2.6860
&1000 & 2.7353 \\
~ & 100000	& 2.5154
&2000	 & 2.4249 \\
\multirow{2}*{16D}
&10000 &	3.8518
&1000	&  3.2668  \\
~ & 100000 &	2.7805
&5000   &	2.9205 \\
\noalign{\smallskip}\hline
\end{tabular}
\end{table}

Generally speaking, not every Neumann problem can be modeled by a loss function without the penalty term on the boundary. Therefore, for a general Neumann problem of the form
\begin{equation}\label{equ:general NBC}
\left\{
  \begin{aligned}
  -\Delta u + \pi^2 u  & = f(x) &x \in \Omega,\\
  \frac{\partial u}{\partial x} & = g(x) & x\in \partial \Omega.
  \end{aligned}
\right.
\end{equation}
we add the penalty term into the loss function \eqref{equ:loss for NBC}
\begin{equation}\label{equ:loss for general}
  \begin{aligned}
  \mathrm{loss} = I(u) +\beta\int_{\partial \Omega} \left(\frac{\partial u}{\partial x} - g(x)\right)^2 \mathrm{d}x.
  \end{aligned}
\end{equation}
The first term is a volume integral while the second term is a boundary integral. Therefore, we have to sample these two terms separately: one
for $I(u)$ in $\Omega$ and the other for the penalty term on the boundary. In principle, we need to optimize sizes of both training sets in order to minimize the approximation error. Practically, we find that QMC method always performs better than MC method. In Table \ref{tab:error of NBC with penalt term}, we show relative $L^2$ errors in different dimensions with different mini-batch sizes for Neumann problem with the penalty term on the boundary.
\begin{table}[htbp]
\centering
\caption{Relative $L^2$ errors in different dimensions with different mini-batch sizes for Neumann problem with the penalty term on the boundary.}
\label{tab:error of NBC with penalt term}       
\begin{tabular}{llllll}
\hline\noalign{\smallskip}
\multirow{2}*{Dimension} &\multicolumn{2}{c}{Mini-batch size} &\multicolumn{2}{c}{Relative $L^2$ error}  \\
~&Volume & Boundary &QMC($\times10^{-2}$) & MC($\times10^{-2}$)\\
\noalign{\smallskip}\hline\noalign{\smallskip}
\multirow{4}*{2D}& 250 & 100&0.8079
 & 1.6151\\
 ~& 500 & 100 & 0.4176 & 1.5027\\
 ~ & 1000 &  100 &0.2086 & 1.0101\\
 ~ & 2000 & 100 & 0.1470 & 0.8327\\
\hline
\multirow{4}*{4D} & 500 & 100&
0.4434 & 1.7622\\
~& 1000 & 100 & 0.3729 & 1.2550\\
~ & 2000 & 100& 0.3160 & 0.9799\\
~ & 4000 &  100& 0.2651 & 0.7530\\
\hline
\multirow{4}*{8D} & 500 & 100& 1.1770 & 3.8285\\
~ & 1000 & 100 & 0.8989 & 3.0505\\
~ & 2000 & 100 & 0.8780 & 2.9661\\
~ & 10000 & 100 & 0.7643 & 1.6552\\
\hline
\multirow{4}*{16D} &1000 &100&1.2181&
3.1398\\
~ & 2000 & 100 &1.1250 &2.2055\\
~ & 5000 & 100 & 0.8998 & 1.9235\\
~ & 10000 & 100 &0.7698 & 1.5244\\
\noalign{\smallskip}\hline
\end{tabular}
\end{table}
Figure \ref{fig:compare 2D NBC with penalty term} plots detailed training processes of QMC and MC methods for different mini-batch sizes for Neumann problem with the penalty term on the boundary in 2D. Size of the training data set for the penalty term is fixed to be $100$ in all cases.
\begin{figure}
     \centering
     \begin{subfigure}[b]{0.48\textwidth}
         \centering
         \includegraphics[width=\textwidth]{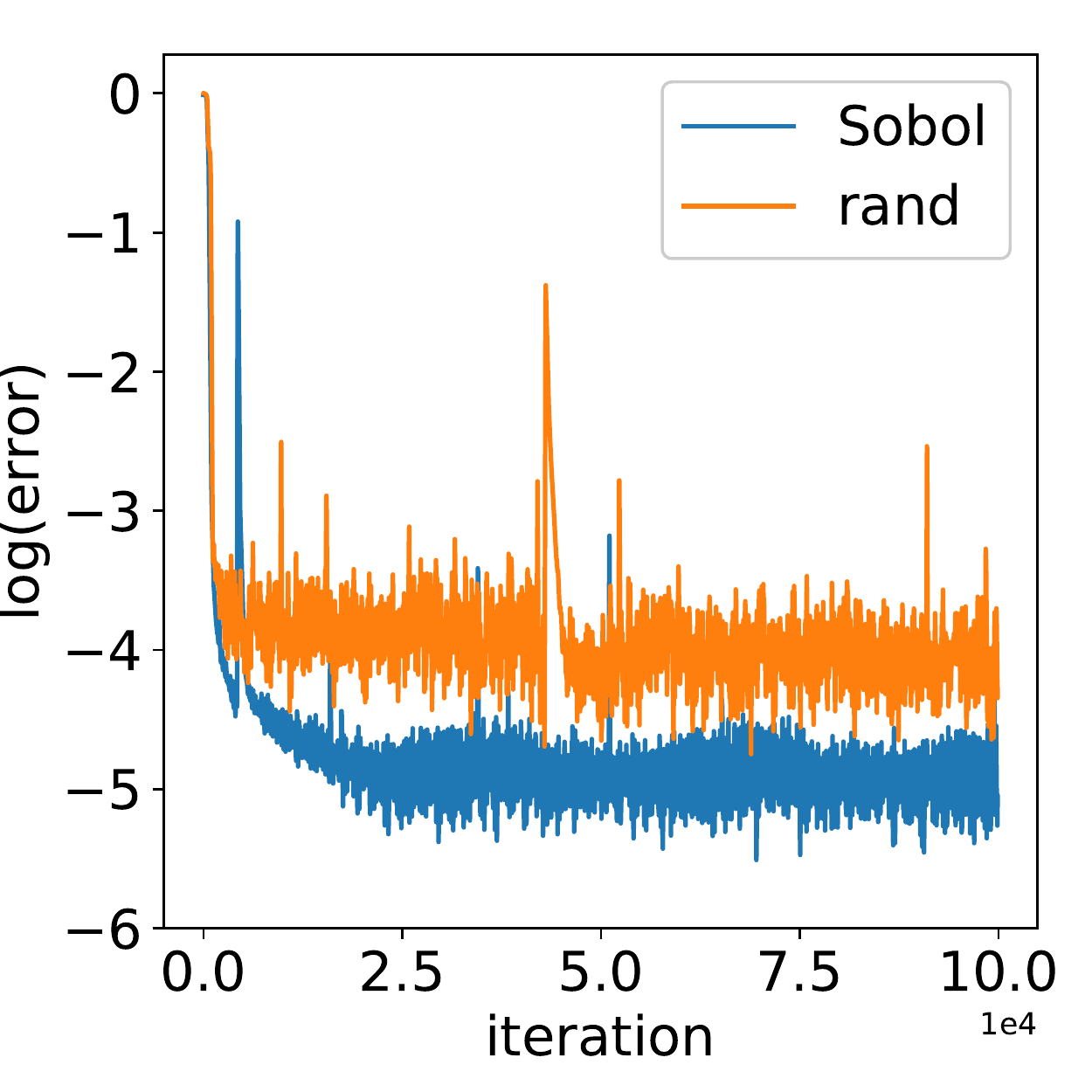}
         \caption{Mini-batch size: $250$}
         \label{fig:compare 2D NBC with penalty term250}
     \end{subfigure}
     \hfill
     \begin{subfigure}[b]{0.48\textwidth}
         \centering
         \includegraphics[width=\textwidth]{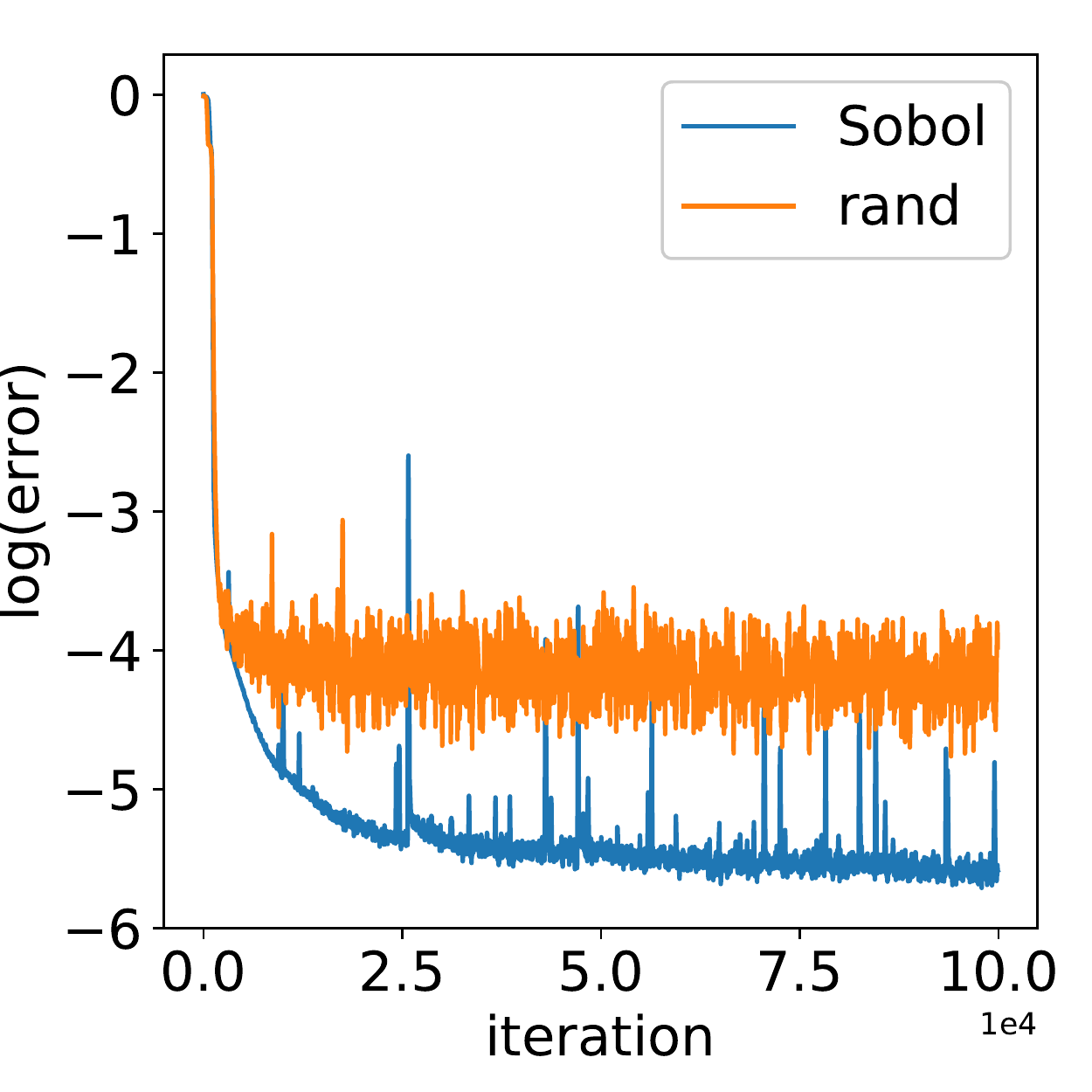}
         \caption{Mini-batch size: $500$}
         \label{fig:compare 2D NBC with penalty term500}
     \end{subfigure}
          \hfill
     \begin{subfigure}[b]{0.48\textwidth}
         \centering
         \includegraphics[width=\textwidth]{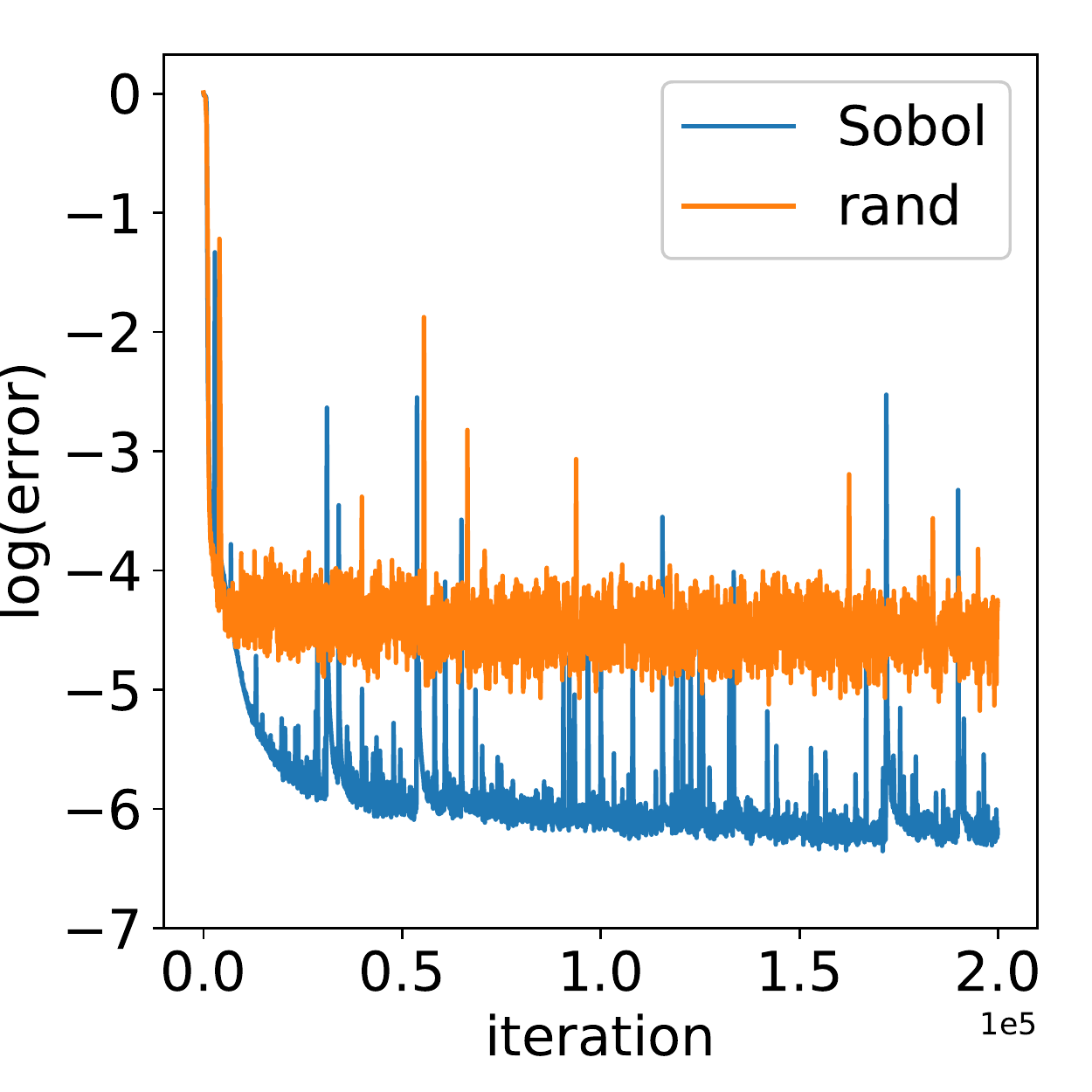}
         \caption{Mini-batch size: $1000$}
         \label{fig:compare 2D NBC with penalty term1000}
     \end{subfigure}
          \hfill
     \begin{subfigure}[b]{0.48\textwidth}
         \centering
         \includegraphics[width=\textwidth]{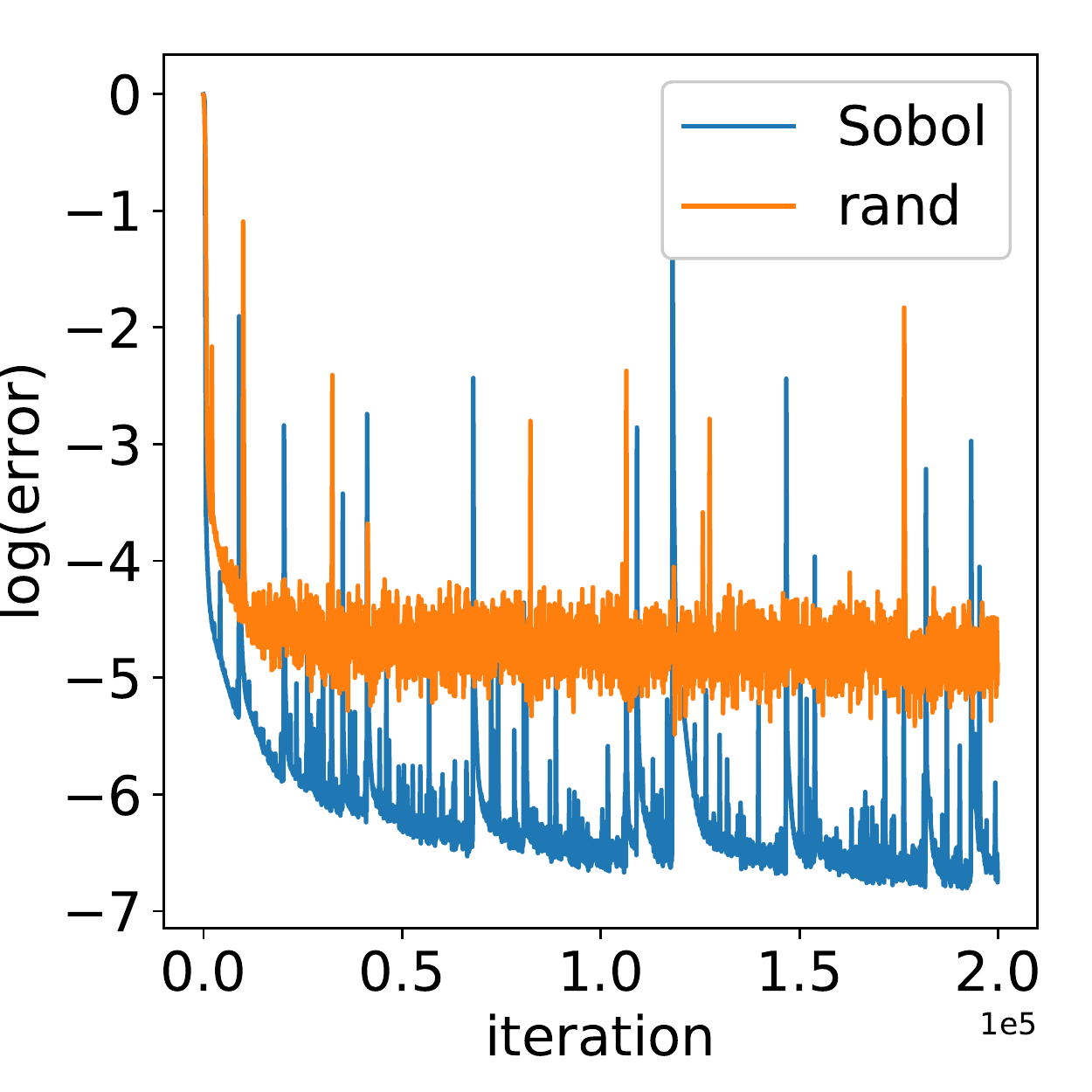}
         \caption{Mini-batch size: $2000$}
         \label{fig:compare 2D NBC with penalty term2000}
     \end{subfigure}
        \caption{Detailed training processes of QMC and MC methods for different mini-batch sizes for Neumann problem with the penalty term over the boundary in 2D. The log function here uses $e$ as base.}
        \label{fig:compare 2D NBC with penalty term}
\end{figure}

\section{Convergence analysis}\label{sec:analysis}

To understand numerical results in \secref{sec:numerical results} from a theoretical perspective, in this section, we shall analyze the convergence behavior of SGD method with QMC sampling and understand how the size of training data set affects the convergence for different sampling strategies. Our analysis follows \cite{Bottou2018optimization}, but differs in terms of assumptions. Under the boundedness assumption of parameter sequence, we prove the Lipschitz continuity of loss function for smooth activation functions, such as \eqref{equ:activation function}, instead of the direct assumption of Lipschitz
continuity of loss function. From a practical perspective, our assumption can be easily verified during the iteration while the Lipschitz continuity is difficult to be verified. Moreover, we prove the convergence of function sequences generated by MC and QMC methods under a stronger assumption on the second moment of the stochastic vector used in the SGD method, which explicitly characterizes the dependence of convergence rate on both the size of training data set and the iteration number.

For convenience, derivatives are taken with respect to the parameter set if subscripts are not specified in what follows. We start with the following assumption.
\begin{assumption}[Boundedness of the iteration sequence]\label{assumption:bound of theta}
  The iteration sequence $\{\theta_i\}_{i=0}^{\infty}$ is bounded, i.e., $\forall i\in\mathbb{N}^+$,
  \begin{equation*}
    \norm{\theta_i}_2 \leq C_{\theta}.
  \end{equation*}
\end{assumption}
A direct consequence of Assumption \ref{assumption:bound of theta} is the existence of convergent sub-sequences. It will be shown later that the whole sequence of function values generated by the SGD method with a sampling strategy converges if the loss function is strongly convex with respect to $\theta$. In practice, this assumption can be verified easily in the iteration. Furthermore, we choose a $C^{\infty}$ activation function $\sigma(x)$; see \eqref{equ:activation function}. Consequently, the residual network also belongs to $C^{\infty}$ with respect to both $\theta$ and $x$. Therefore, the Lipschitz continuity of the approximate solution by the neural network is guaranteed in a bounded domain.
\begin{lem}\label{lem:L-continuous gradient}
Under Assumption \ref{assumption:bound of theta}, the $C^{\infty}$ activation function, and the bounded domain $\Omega$, there exist constants $L_0$, $L_1$ and $L_2$, $\forall  \theta, \bar{\theta}\in\{\theta_i\}_{i=1}^{\infty}$ and $x\in\Omega$ such that
\begin{equation*}
  \begin{aligned}
    \norm{\nabla\hat{u}_{\theta} - \nabla\hat{u}_{\bar{\theta}}}_2 & \leq L_0\norm{\theta - \bar{\theta}}_2, \\
    \norm{\nabla_x\hat{u}_{\theta} - \nabla_x\hat{u}_{\bar{\theta}}}_2 & \leq L_1\norm{\theta - \bar{\theta}}_2, \\
    \norm{\nabla\nabla_x\hat{u}_{\theta} - \nabla\nabla_x\hat{u}_{\bar{\theta}}}_2 & \leq L_2\norm{\theta - \bar{\theta}}_2.
  \end{aligned}
\end{equation*}
\end{lem}

A natural corollary of \lemref{lem:L-continuous gradient} is that there exist constants $M_1$ and $M_2$, $\forall \theta\in\{\theta_i\}_{i=1}^{\infty}$ and $x\in\Omega$, such that
\begin{equation}\label{equ:nabla_x bound}
  \norm{\nabla_x\hat{u}_{\theta}(x)}  \leq M_1, \quad  \norm{\nabla\nabla_x\hat{u}_{\theta}(x)}_2  \leq M_2.
\end{equation}
Remember that above results hold for smooth activation functions, but does not hold for ReLU activation function. Then we can only expect the boundness of $\nabla_x\hat{u}_{\theta}(x)$ and $\nabla\nabla_x\hat{u}_{\theta}(x)$, instead of Lipschitz continuity. For example, consider the simplest network with two layers. For the ReLU activation function \cite{nair2010rectified}, there exists a constant $M$ depending on $\Omega$ such that
\begin{equation}\label{equ:L-continuous ReLU}
  \norm{\nabla\hat{u}_{\theta} - \nabla\hat{u}_{\bar{\theta}}}_2 \leq L_0\norm{\theta - \bar{\theta}}_2 + M.
\end{equation}

Based on Lipschitz continuity of the neural network solution and boundedness of derivatives, Lipschitz continuity of the loss function can be proven.
\begin{thm}[Lipschitz continuity of the loss function]\label{thm:L-continuous}
  For any $\theta$ and $\bar{\theta}$, the loss function of the neural network satisfies
  \begin{equation}
    \norm{\nabla I(\theta) - \nabla I(\bar{\theta})}_2 \leq L\norm{\theta - \bar{\theta}}_2.
  \end{equation}
\end{thm}
\begin{proof}
From \eqref{equ:dnnsolution}, we have
  \begin{align*}
    \nabla I(\theta) = \int_{\Omega}\nabla\nabla_x\hat{u}_{\theta}(x)\cdot\nabla_x\hat{u}_{\theta}(x) \mathrm{d}x - \int_{\Omega}f(x)\nabla\hat{u}_{\theta}(x)\mathrm{d}x.
  \end{align*}
By the triangle inequality, we have
\begin{align*}
    \norm{\nabla I(\theta) - \nabla I(\bar{\theta})}_2 & \leq \norm{\int_{\Omega}\left(\nabla\nabla_x\hat{u}_{\theta}(x)\cdot\nabla_x\hat{u}_{\theta}(x) - \nabla\nabla_x\hat{u}_{\theta}(x)\cdot\nabla_x\hat{u}_{\bar{\theta}}(x)\right)\mathrm{d}x}_2 \\
    &\quad + \norm{\int_{\Omega}\left(\nabla\nabla_x\hat{u}_{\theta}(x)\cdot\nabla_x\hat{u}_{\bar{\theta}}(x) - \nabla\nabla_x\hat{u}_{\bar{\theta}}(x)\cdot\nabla_x\hat{u}_{\bar{\theta}}(x)\right)\mathrm{d}x}_2 \\
    & \quad + \norm{\int_{\Omega}f(x)\left(\nabla\hat{u}_{\theta}(x) - \nabla\hat{u}_{\bar{\theta}}(x)\right)\mathrm{d}x}_2\\
    & \leq M_2L_1\norm{\theta - \bar{\theta}}_2 + M_1L_2\norm{\theta - \bar{\theta}}_2 + L_0\max_{x\in\Omega}\abs{f(x)}\norm{\theta - \bar{\theta}}_2\\
    & \leq L\norm{\theta - \bar{\theta}}_2
  \end{align*}
  with $L = M_2L_1 + M_1L_2 + L_0\max_{x\in\Omega}\abs{f(x)}$.
\end{proof}

An important consequence of the \thmref{thm:L-continuous} is that for $\forall\theta,\bar{\theta}\in\{\theta_i\}_{i=1}^{\infty}$,
\begin{equation}
  \label{equ:L-continuous direct}
  I(\theta) \leq I(\bar{\theta}) + \nabla I(\bar{\theta})^T(\theta-\bar{\theta}) + \frac 1{2}L\norm{\theta-\bar{\theta}}_2^2.
\end{equation}

To proceed, we need the following assumptions.
\begin{assumption}[First and second moment assumption]\label{assumption:moment}
For the stochastic vector, assume that first and second moments satisfy
  \begin{itemize}
    \item There exists $0<\mu\leq\mu_{G}$ such that, $\forall i\in\mathbb{N}^+$,
    \begin{equation}
      \nabla I(\theta_i)^T\mathds{E}_{\xi_i}[g(\theta_i, \xi_i)] \geq \mu(1 + r(N))\norm{\nabla I(\theta_i)}_2^2
    \end{equation}
    and
    \begin{equation}
      \norm{\mathds{E}_{\xi_i}[g(\theta_i, \xi_i)]}_2 \leq \mu_G(1 + r(N))\norm{\nabla I(\theta_i)}_2.
    \end{equation}
    \item For the second moment, there exist $C_V\geq0$ and $M_V\geq0$ such that, $\forall i\in\mathbb{N}^+$,
    \begin{equation}\label{equ:variance limit}
      \mathds{V}_{\xi_i}[g(\theta_i, \xi_i)] \leq C_Vr(N) + M_V(1+r(N))\norm{\nabla I(\theta_i)}_2^2.
    \end{equation}
  \end{itemize}
  where $r(N) = \left(D(P_N)\right)^2$, and $r(N)=O(1/N)$ for MC sampling and $r(N) = O((\ln(N))^{2K}/N^2)$ for QMC sampling.
\end{assumption}
Note that $g(\theta_i, \xi_i)$ is an unbiased estimate of $\nabla I(\theta_i)$ as $N\rightarrow\infty$ when  $\mu=\mu_G = 1$ in Assumption \ref{assumption:moment}. Moreover, according to \eqref{equ:variance limit}, we have
\begin{align*}
  \mathds{E}_{\xi_i}[\norm{g(\theta_i, \xi_i)}_2^2] &= \mathds{V}_{\xi_i}[g(\theta_i, \xi_i)] + \norm{\mathds{E}_{\xi_i}[g(\theta_i, \xi_i)]}_2^2 \\
  &\leq C_Vr(N) + M_G(1+r(N))\norm{\nabla I(\theta_i)}_2^2
\end{align*}
with $M_G = M_V + \mu_G^2(1+r(N))$. For a fixed stepsize $\alpha_i = \alpha$, according to \eqref{equ:L-continuous direct}, we have
\begin{align}
  \label{equ:inequality of expectation}
  \mathds{E}_{\xi_i}[I(\theta_{i+1})] - I(\theta_i) \leq&\frac L{2}\alpha^2C_Vr(N) - \alpha\left(\mu - \frac L{2}\alpha M_G\right)\left(1+r(N)\right)\norm{\nabla I(\theta_i)}_2^2.
\end{align}

\begin{assumption}[Strong convexity]\label{assumption:strong convexity}
The loss function $I(\theta)$ is strongly convex with respect to $\theta$, i.e., there exists a constant $c$, $\forall \theta, \bar{\theta}\in\{\theta_i\}_{i=1}^{\infty}$, such that
  \begin{equation*}
    I(\bar{\theta}) \geq I(\theta) + \nabla I(\theta)^T(\bar{\theta} - \theta) + \frac {c}{2}\norm{\bar{\theta} - \theta}_2^2.
  \end{equation*}
\end{assumption}
Following \cite{Bottou2018optimization}, if $I(\theta)$ is strongly convex and $\theta^*$ is the unique minimizer, $\forall \theta\in\{\theta_i\}_{i=1}^{\infty}$, we have
\begin{equation}\label{equ:strongconvexity}
2c(I(\theta) - I(\theta^*)) \leq \norm{\nabla I(\theta)}_2^2
\end{equation}
with $c\leq L$.
\begin{thm}[Strongly convex loss function, fixed stepsize, MC and QMC samplings]
\label{thm:the last thm} Let the data set $P_N = \{x^1,x^2,\cdots, x^{N} \}\subset\Omega$ being generated by MC method or QMC method, the stepsize $\alpha_i = \alpha$ being a constant satisfying
\[
0 < \alpha \leq \frac{\mu}{LM_G}.
\]
Under Assumptions \ref{assumption:bound of theta}, \ref{assumption:moment}, and \ref{assumption:strong convexity}, the expected optimality gap satisfies
\begin{align}\label{equ:convergence}
  \lim\limits_{i\to\infty}\mathds{E}[I(\theta_i) - I(\theta^*)] = \frac{\alpha LC_V}{2c\mu}r(N)
\end{align}
with $\theta^*$ being the unique minimizer.
\end{thm}
\begin{proof}
Using \eqref{equ:inequality of expectation} and \eqref{equ:strongconvexity}, $\forall i\in\mathbb{N}$, we have
  \begin{align*}
    \mathds{E}_{\xi_i}[I(\theta_{i+1})] - I(\theta_i) &\leq \frac L{2}\alpha^2C_Vr(N) - \alpha\left(\mu - \frac L{2}\alpha M_G\right)\left(1+r(N)\right)\norm{\nabla I(\theta_i)}_2^2 \\
    &\leq \frac L{2}\alpha^2C_Vr(N) - \frac{\mu\alpha}{2}\left(1+r(N)\right)\norm{\nabla I(\theta_i)}_2^2\\
    &\leq \frac L{2}\alpha^2C_Vr(N) - c\mu\alpha(1+r(N))(I(\theta_i) - I(\theta^*))\\
    &\leq \frac L{2}\alpha^2C_Vr(N) - c\mu\alpha(I(\theta_i) - I(\theta^*)),
  \end{align*}
since $I(\theta_i) - I(\theta^*)\geq 0$ holds under the assumption of strongly convexity.
Adding $I(\theta_i) - I(\theta^*)$ to both sides of the above inequality and taking total expectation yields
  \begin{align*}
    \mathds{E}[I(\theta_{i+1}) - I(\theta^*)] \leq (1-c\mu\alpha)\mathds{E}[I(\theta_i) - I(\theta^*)] + \frac L{2}\alpha^2C_Vr(N).
  \end{align*}
Subtracting $\frac{\alpha LC_V}{2c\mu}r(N)$ from both sides produces
  \begin{align*}
   \mathds{E}[I(\theta_{i+1}) - I(\theta^*)] - \frac{\alpha LC_V}{2c\mu}r(N) \leq (1-c\mu\alpha)\left(\mathds{E}[I(\theta_{i}) - I(\theta^*)] - \frac{\alpha LC_V}{2c\mu}r(N) \right).
  \end{align*}
Note that
  \[
  0 < c\mu\alpha \leq \frac{c\mu^2}{LM_G} \leq \frac{c\mu^2}{L\mu^2} = \frac{c}{L} \leq 1,
  \]
a recursive argument yields
\begin{align*}
\mathds{E}[I(\theta_{i+1}) - I(\theta^*)] - \frac{\alpha LC_V}{2c\mu}r(N) \leq (1-c\mu\alpha)^i\left(\mathds{E}[I(\theta_{1}) - I(\theta^*)] - \frac{\alpha LC_V}{2c\mu}r(N) \right),
\end{align*}
and thus
\begin{align*}
    \lim\limits_{i\to\infty}\mathds{E}[I(\theta_i) - I(\theta^*)] = \frac{\alpha LC_V}{2c\mu}r(N).
\end{align*}
\end{proof}

\thmref{thm:the last thm} provides a quantitative estimate of the accuracy of sampling strategies on the convergence rate of the iteration sequence. Regardless of the sampling strategy, the convergence rate is linear which is determined by the SGD method.
\begin{rem}
  When $N\to\infty$, for both MC and QMC methods, we have $r(N)\to0$. Form \thmref{thm:the last thm}, we conclude that
  \begin{equation*}
    \lim\limits_{N\to\infty}\lim\limits_{i\to\infty}\mathds{E}[I(\theta_i) - I(\theta^*)] = 0.
  \end{equation*}
In practice, however, with the increasing size of training data set, the gap usually does not tend to $0$ under the fixed stepsize condition.
We attributes this to the irrationality of \eqref{equ:variance limit}. Instead, \eqref{equ:variance limit} shall be relaxed to \cite{Bottou2018optimization}
  \begin{equation*}
      \mathds{V}_{\xi_i}[g(\theta_i, \xi_i)] \leq C_V + M_V(1+r(N))\norm{\nabla I(\theta_i)}_2^2.
    \end{equation*}
Then we have
  \begin{equation*}
    \lim\limits_{N\to\infty}\lim\limits_{i\to\infty}\mathds{E}[I(\theta_i) - I(\theta^*)] = \frac{\alpha LC_V}{2c\mu},
  \end{equation*}
which implies the convergence of the sequence of function values near the optimal value. Therefore, an optimizer with fixed stepsize is generally not the best choice \cite{Bottou2018optimization}. Instead, the SGD method with diminishing stepsize is popular in real applications, such as ADAM method \cite{ADAM} using in the implementation.
\end{rem}

\begin{rem}
From the convergence analysis \eqref{equ:convergence}, we expect that QMC method outperforms MC method in terms of convergence order. It is reasonable to find from Tables \ref{tab:convergence error of DBC} and \ref{tab:convergence error of NBC} that QMC method has better rates than MC method. However, the difference in rate is not as significant as the difference in magnitude. We attribute this difference to the nonconvexity of the loss function and the above analysis relies crucially on the strong convexity assumption of the same function. Whatever, for the same accuracy requirement in practice, QMC method outperforms MC method in terms of efficiency by orders of magnitude. Therefore, we recommend the usage of QMC sampling in machine-learning PDEs whenever a sampling strategy is needed.
\end{rem}

\section{Conclusion}\label{sec:conclusion}
In this paper, we have proposed to approximate the loss function using quasi-Monte Carlo method, instead of Monte Carlo method that is commonly used in machine-learning PDEs. Numerical results based on deep Ritz method have shown the significant advantage of quasi-Monte Carlo method in terms of accuracy or efficiency. All the codes that generate numerical results included in this work are available from \url{https://github.com/Lyupinpin/DeepRitzMethod}.

Theoretically, we have proved the convergence of neural network solver based on quasi-Monte Carlo sampling in terms of the sampling size and the iteration number. Although there are practical issues such as the nonconvexity of the loss function, our analysis does provide a comprehensive understanding of why quasi-Monte Carlo method always outperforms Monte-Carlo method and suggests the usage of the former whenever an approximation of high-dimensional integrals is needed.

\section*{Acknowledgements}
This work is supported in part by the grants NSFC 21602149 and 11971021 (J.~Chen), NSFC 11501399 (R.~Du). P.~Li is grateful to Ditian Zhang for very helpful discussions. L.~Lyu acknowledges the financial support of Undergraduate Training Program for Innovation and Entrepreneurship, Soochow University (Projection 201810285019Z). Part of the work was done when L.~Lyu was doing a summer internship at Department of Mathematics, Hong Kong University of Science and Technology. L.~Lyu would like to thank its hospitality.
%
%

\normalem
\bibliographystyle{unsrt}
\bibliography{refs}

\begin{thebibliography}{10}

\bibitem{E2018}
Weinan E and Bing Yu.
\newblock The deep ritz method: A deep learning-based numerical algorithm for
  solving variational problems.
\newblock {\em Communications in Mathematics and Statistics}, 6(1):1--12, 2018.

\bibitem{Goodfellow2016}
Ian Goodfellow, Yoshua Bengio, and Aaron Courville.
\newblock {\em Deep Learning}.
\newblock MIT Press, 2016.

\bibitem{Wang2012EndtoendTR}
Tao Wang, David~J. Wu, Adam Coates, and Andrew~Y. Ng.
\newblock End-to-end text recognition with convolutional neural networks.
\newblock {\em Proceedings of the 21st International Conference on Pattern
  Recognition (ICPR2012)}, pages 3304--3308, 2012.

\bibitem{NIPS2012_4824}
Alex Krizhevsky, Ilya Sutskever, and Geoffrey~E. Hinton.
\newblock Imagenet classification with deep convolutional neural networks.
\newblock {\em Communications of the ACM}, 60:84--90, 2012.

\bibitem{hinton2012deep}
Hinton Geoffrey, Deng Li, Yu~Dong, Dahl George, Mohamed Abdel-rahman, Jaitly
  Navdeep, Senior Andrew, Vanhoucke Vincent, Nguyen Patrick, Kingsbury Brian,
  and Sainath Tara.
\newblock Deep neural networks for acoustic modeling in speech recognition.
\newblock {\em IEEE Signal Processing Magazine}, 29:82--97, 2012.

\bibitem{Sarikaya:2014:ADB:2687012.2687014}
Ruhi Sarikaya, Geoffrey~E. Hinton, and Anoop Deoras.
\newblock Application of deep belief networks for natural language
  understanding.
\newblock {\em IEEE/ACM Transactions on Audio, Speech, and Language
  Processing}, 22(4):778--784, 2014.

\bibitem{Jiequn2018}
Jiequn Han, Arnulf Jentzen, and Weinan E.
\newblock Solving high-dimensional partial differential equations using deep
  learning.
\newblock {\em PNAS}, 115(34):8505--8510, 2018.

\bibitem{Fan2019}
Yuwei Fan, Feliu-Fab{\`a} Jordi, Lin Lin, Lexing Ying, and Leonardo
  Zepeda-N{\'u}{\~{n}}ez.
\newblock A multiscale neural network based on hierarchical nested bases.
\newblock {\em Research in the Mathematical Sciences}, 6(2):21, 2019.

\bibitem{deepGalerkin2018}
Justin Sirignano and Konstantinos Spiliopoulos.
\newblock {DGM: A} deep learning algorithm for solving partial differential
  equations.
\newblock {\em Journal of Computational Physics}, 375:1339--1364, 2018.

\bibitem{Carleo2017}
Carleo Giuseppe and Troyer Matthias.
\newblock Solving the quantum many-body problem with artificial neural
  networks.
\newblock {\em Science}, 355(6325):602--606, 2017.

\bibitem{Han2019}
Jiequn Han, Linfeng Zhang, and Weinan E.
\newblock Solving many-electron schr\"odinger equation using deep neural
  networks.
\newblock {\em Journal of Computational Physics}, 399, 2019.

\bibitem{Weinan2017349}
Weinan E, Jiequn Han, and Arnulf Jentzen.
\newblock Deep learning-based numerical methods for high-dimensional parabolic
  partial differential equations and backward stochastic differential
  equations.
\newblock {\em Communications in Mathematics and Statistics}, 5(4):349--380,
  2017.

\bibitem{Beck20191563}
Christian Beck, Weinan E, and Arnulf Jentzen.
\newblock Machine learning approximation algorithms for high-dimensional fully
  nonlinear partial differential equations and second-order backward stochastic
  differential equations.
\newblock {\em Journal of Nonlinear Science}, 29(4):1563--1619, 2019.

\bibitem{GonzlezCervera2019}
J~A~Gonz{\'{a}}lez Cervera.
\newblock Solution of the black-scholes equation using artificial neural
  networks.
\newblock {\em Journal of Physics: Conference Series}, 1221:012044, 2019.

\bibitem{Randall2007}
Randall~J. LeVeque.
\newblock {\em Finite Difference Methods for Ordinary and Partial Differential
  Equations: Steady-State and Time-Dependent Problems}.
\newblock Society for Industrial and Applied Mathematics, 2007.

\bibitem{FEM2007}
Susanne~C. Brenner and L.~Ridgway Scoot.
\newblock {\em The Mathematical Theory of Finite Element Method}.
\newblock Springer, 2007.

\bibitem{spectral2011}
Jie Shen, Tao Tang, and Li-Lian Wang.
\newblock {\em Spectral methods: Algorithm, Analysis and Application}.
\newblock Springer, 2011.

\bibitem{gerstner_numerical_1998}
Thomas Gerstner and Michael Griebel.
\newblock Numerical integration using sparse grids.
\newblock {\em Numerical Algorithms}, 18(3-4):209, January 1998.

\bibitem{bungartz_sparse_2004}
Hans-Joachim Bungartz and Michael Griebel.
\newblock Sparse grids.
\newblock {\em Acta Numerica}, 13:147--269, May 2004.

\bibitem{DG2016}
Xiaobing Feng, Thomas Lewis, and Michael Neilan.
\newblock Discontinuous {Galerkin} finite element differential calculus and
  applications to numerical solutions of linear and nonlinear partial
  differential equations.
\newblock {\em Journal of Computational and Applied Mathematics}, 299:68--91,
  2016.

\bibitem{Bottou2018optimization}
L\'eon Bottou, Frank~E. Curtis, and Jorge Nocedal.
\newblock Optimization for large-scale machine learning.
\newblock {\em SIAM Review}, 60(2):223--311, 2018.

\bibitem{liu2008monte}
Jun~S Liu.
\newblock {\em Monte Carlo strategies in scientific computing}.
\newblock Springer Science \& Business Media, 2008.

\bibitem{Ogata1989}
Yosihiko Ogata.
\newblock A {Monte Carlo} method for high dimensional integration.
\newblock {\em Numerische Mathematik}, 55(2):137--157, 1989.

\bibitem{QMCerror1992}
H.~Niederreiter.
\newblock Random number generation and quasi-{Monte Cralo} methods.
\newblock {\em CBMS-SIAM}, 63, 1992.

\bibitem{Dick2013QMC}
Josef Dick, Frances~Y. Kuo, and Ian~H. Sloan.
\newblock High-dimensional integration: The quasi-{Monte Carlo} way.
\newblock {\em Acta Numerical}, 22:133--288, 2013.

\bibitem{Bruno2004survey}
Bruno Tuffin.
\newblock Randomization of {Quasi-Monte Carlo} methods for error estimation:
  survey and normal approximation.
\newblock {\em Monte Carlo Methods and Applications}, 10(3):617--628, 2004.

\bibitem{QMCvariational2018}
Alexander Buchholz, Florian Wenzel, and Stephan Mandt.
\newblock {Quasi-Monte Carlo} variational inference.
\newblock {\em arXiv}, 1807.01604, 2018.

\bibitem{Liyao2019}
Liyao Lyu, Zhiwen Zhang, and Jingrun Chen.
\newblock Connecting exciton diffusion with surface roughness via deep
  learning.
\newblock {\em arXiv}, 1910.14209, 2019.

\bibitem{evans_2010}
Evans~C. Lawrence.
\newblock {\em Partial differential equations(second edition)}.
\newblock American Mathematical Society, 2010.

\bibitem{DBLP:journals/corr/abs-1710-05941}
Prajit Ramachandran, Barret Zoph, and Quoc~V. Le.
\newblock Searching for activation functions.
\newblock {\em arXiv}, 1710.05941, 2017.

\bibitem{DBLP:journals/corr/HeZRS15}
Kaiming He, Xiangyu Zhang, Shaoqing Ren, and Jian Sun.
\newblock Deep residual learning for image recognition.
\newblock {\em 2016 IEEE Conference on Computer Vision and Pattern Recognition
  (CVPR)}, 2:770--778, 2016.

\bibitem{niederreiter1992random}
Harald Niederreiter.
\newblock {\em Random number generation and quasi-{Monte Carlo} methods}.
\newblock SIAM, 1992.

\bibitem{Russel1998}
Russel Caflisch.
\newblock {\em Monte Carlo and quasi-{Monte Carlo} methods}, volume~7.
\newblock Cambridge University Press.

\bibitem{dick2013high}
Josef Dick, Frances~Y. Kuo, and Ian~H. Sloan.
\newblock High-dimensional integration: the quasi-{Monte Carlo} way.
\newblock {\em Acta Numerica}, 22:133--288, 2013.

\bibitem{Sobol1976}
I.~M. Sobol.
\newblock Uniformly distributed sequences with an additional uniform property.
\newblock {\em USSR Computational Mathematics and Mathematical Physics},
  16(5):236--242, 1976.

\bibitem{ADAM}
Diederik~P. Kingma and Jimmy Ba.
\newblock Adam: A method for stochastic optimization.
\newblock {\em CoRR}, 1412.6980, 2014.

\bibitem{nair2010rectified}
Vinod Nair and Geoffrey~E Hinton.
\newblock Rectified linear units improve restricted boltzmann machines.
\newblock In {\em Proceedings of the 27th international conference on machine
  learning (ICML-10)}, pages 807--814, 2010.

\end{thebibliography}

\end{document}